\documentclass{birkjour}

\usepackage[utf8]{inputenc}
\usepackage[T1]{fontenc}

\usepackage{graphicx}		
\usepackage{mathscinet}		
\usepackage{enumitem}		
\usepackage{color}			
\usepackage{hyperref}		
\hypersetup{hidelinks}		
\usepackage{pdflscape}		

\usepackage{amssymb, amsmath,amsfonts,latexsym}
\usepackage{empheq}

\usepackage{booktabs}
\usepackage{tabulary}
\usepackage{tabularx}
\usepackage{tablefootnote}
\usepackage{longtable}

\usepackage[font=small,floatrowsep=qquad,captionskip=5pt]{floatrow}
\newtheorem{Theorem}{Theorem}[section]

\newtheorem{Lemma}[Theorem]{Lemma}
\newtheorem{Proposition}[Theorem]{Proposition}
\theoremstyle{definition}
\newtheorem{Definition}[Theorem]{Definition}
\theoremstyle{remark}
\newtheorem{Remark}[Theorem]{Remark}
\newtheorem*{Example}{Example}
\numberwithin{equation}{section}

\newcommand{\N}{{\mathbb{N}}}  
\newcommand{\Z}{{\mathbb{Z}}}  
\newcommand{\R}{{\mathbb{R}}}  
\newcommand{\Czi}{{C^{\infty}_{0}}}  
\newcommand{\Sw}{\mathcal{S}}	

\newcommand{\db}{\normalfont{\text{\dbar}}}
\newcommand{\rmd}{\mathrm{d}}
\newcommand{\I}{\mathrm{i}}

\newcommand{\jbl}{\langle}
\newcommand{\jbr}{\rangle}

\newcommand{\vp}{\varphi}
\newcommand{\ve}{\varepsilon}

\newcommand{\Osii}{\mathrm{Os-}\iint\limits_{\R^{2n}}}
\newcommand{\Sy}{\mathcal{S}}

\newcommand{\OPS}{\Psi}
\newcommand{\OP}{Op}
\newcommand{\jxi}{\jbl\xi\jbr}
\newcommand{\B}{B}

\newcommand{\W}{\mathcal W}
\newcommand{\Zhyp}{Z_{\text{hyp}}}
\newcommand{\Zpd}{Z_{\text{pd}}}
\newcommand{\T}{\mathrm{T}}
\renewcommand{\S}{\mathrm{S}}

\newcommand{\chixi}{\chi\Big(\frac{\jxi}{N (w(\Lambda(t)))^m}\Big)}

\renewcommand{\Re}{\operatorname{Re}}

\DeclareMathOperator*{\supp}{supp}

\newcommand*\xbar[1]{%
	\hbox{%
		\vbox{%
			\hrule height 0.5pt 
			\kern0.5ex
			\hbox{%
				\kern-0.1em
				\ensuremath{#1}%
				\kern-0.1em
			}%
		}%
	}%
}

\allowdisplaybreaks

\begin{document}

	
	\title[A generalized Levi condition for weakly hyperbolic Cauchy problems]{A generalized Levi condition for weakly hyperbolic Cauchy problems with coefficients low regular in time and smooth in space}
	\author[Lorenz]{Daniel Lorenz}
	\address{%
		TU Bergakademie Freiberg, Faculty of Mathematics and Computer Science\\Institute of Applied Analysis\\Pr{\"u}ferstraße 9,\\09599 Freiberg, Germany
	}
	\email{daniel.lorenz@math.tu-freiberg.de}
	\author[Reissig]{Michael Reissig}
	\address{%
		TU Bergakademie Freiberg, Faculty of Mathematics and Computer Science\\Institute of Applied Analysis\\Pr{\"u}ferstraße 9,\\09599 Freiberg, Germany
	}
	\email{reissig@math.tu-freiberg.de}

	\subjclass{35S05, 35L30, 47G30}
	\keywords{weakly hyperbolic, Cauchy problem, Levi condition, modulus of continuity, weight sequence, weight function}

	\begin{abstract}
		We consider the Cauchy problem for weakly hyperbolic $m$-th order partial differential equations with coefficients low-regular in time and smooth in space.
		It is well-known that in general one has to impose Levi conditions to get $C^\infty$ or Gevrey well-posedness even if the coefficients are smooth. We use moduli of continuity to describe the regularity of the coefficients with respect to time, weight sequences for the characterization of their regularity with respect to space and weight functions to define the solution spaces. Furthermore, we propose a generalized Levi condition that models the influence of multiple characteristics more freely. We establish sufficient conditions for the well-posedness of the Cauchy problem, that link the Levi condition as well as the modulus of continuity and the weight sequence of the coefficients to the weight function of the solution space. Additionally, we obtain that the influences of the Levi condition and the low regularity of coefficients on the weight function of the solution space are independent of each other.
	\end{abstract}
	
	\maketitle
	
	\section{Introduction}\label{INTRO}
	
	In this paper, we consider the weakly hyperbolic Cauchy problem
	\begin{equation}\label{INTRO:CP}
	\begin{aligned}
	D_t^m u = \sum\limits_{j=0}^{m-1} \sum\limits_{|\gamma|+j = m} \lambda(t)^{m-j}a_{m-j,\,\gamma}(t,\,x) D_x^\gamma D_t^j u\\
	 + \sum\limits_{|\gamma|+j < m} b_{m-j,\,\gamma}(t,\,x) D_x^\gamma D_t^j u,\\
	D_t^{k-1}u(0,\,x) = g_k(x),\,k = 1,\,\ldots,\,m,\,(t,\,x) \in [0,\,T]\times\R^n,
	\end{aligned}
	\end{equation}
	and we investigate how the interplay of low regularity of the coefficients with respect to time and weak hyperbolicity of the problem (due to a degeneration at $t = 0$) influences the well-posedness and possible solution spaces of the above problem.
	
	Historically, the effects of low-regular coefficients and multiple characteristics have been studied quite extensively, when just one of the two effects is present.
	
	Let us first recall some results for strictly hyperbolic equations with low-regular coefficients and then turn our attention to results for weakly hyperbolic equations.
	
	For strictly hyperbolic Cauchy problems, it is well-known that there is (in general) no $C^\infty$ or $H^\infty$ well-posedness, if the regularity of the coefficients in time is below Lipschitz.
	Usually, working with low-regular coefficients requires higher regularity in space for the solution and initial data to obtain well-posedness.
	
	For second-order equations with $t$-dependent coefficients, in \cite{Colombini.1979} the authors proved well-posedness for H{\"o}lder continuous coefficients, with Gevrey regularity in space for the solution and initial data.
	
	\cite{Nishitani.1983} and \cite{Jannelli.1985} were able to extend the results of \cite{Colombini.1979} by working with $(t,\,x)$-dependent coefficients and assuming H{\"o}lder regularity in time and Gevrey in the spatial variables of the coefficients.
		
	Considering equations of order $m$ with Log-Lipschitz and H{\"o}lder continuous coefficients, which also depend on $x$, \cite{Cicognani.1999} extended the results of \cite{Jannelli.1985, Nishitani.1983}.
	
	More recently, in \cite{Cicognani.2017} the authors established a general condition linking the regularity of the coefficients in time to possible solution spaces and the required regularity of the coefficients in space. They assumed that the coefficients satisfy the relation
	\begin{equation}\label{INTRO:coeff}
	\begin{aligned}[m]
	\big|D_x^\beta a_{m-j,\,\gamma}(t,\,x) - D_x^\beta a_{m-j,\,\gamma}(s,\,x)\big| \leq C K_{|\beta|} \mu(|t-s|),\\ 0 \leq |t-s| \leq 1, \,x\in\R^n,
	\end{aligned}
	\end{equation}
	where $\mu$ is a modulus of continuity describing their regularity in time and $\{K_p\}_p$ is a weight sequence describing their regularity in space. Under suitable assumptions on the weight sequence $\{K_p\}_p$ and a weight function $\eta$, they proved well-posedness in spaces
	\begin{equation*}
	H^\nu_{\eta,\,\delta}(\R^n) = \Big\{f \in \Sw^\prime(\R^n)\,|\, \jbl D_x\jbr^\nu e^{\delta\eta(\jbl D_x \jbr)} f(x) \in L^2(\R^n)\Big\},\,\delta>0,
	\end{equation*}
	provided that $\mu(\jxi^{-1}) \jxi = o(\eta(\jxi))$.
	
	Our approach makes use of the results of \cite{Cicognani.2017} for strictly hyperbolic problems. In this paper, we work with the spaces $H^\nu_{\eta,\,\delta} = H^\nu_{\eta,\,\delta}(\R^n)$ and we also assume that the coefficients of the principal part are $\mu$-continuous and satisfy \eqref{INTRO:coeff}.
	
	Let us, for completeness, recall what we understand by the term modulus of continuity.
	\begin{Definition}[Modulus of Continuity and $\mu$-Continuity]\label{SH:BM:RegTime:Definition:MOC}
		We call $\mu: [0,\,1] \rightarrow [0,\,1]$ a modulus of continuity, if $\mu$ is continuous, concave and increasing and satisfies $\mu(0) = 0$.
		A function $f \in C(\R^n)$ belongs to $C^\mu(\R^n)$ if and only if
		\begin{equation*}
		|f(x) - f(y)| \leq C \mu(|x-y|),
		\end{equation*}
		for all $x,\,y \in \R^n,\, |x-y| \leq 1$ and some constant $C$.
		
		Typical examples of moduli of continuity are presented in the Table~\ref{table1}.
	\end{Definition}
		\begin{table}[h!]
		\begin{tabulary}{\textwidth}{LL}
			\toprule
			modulus of continuity &commonly called\\
			\midrule
			$\mu(s) = s$	& Lipschitz-continuity\\[5pt]
			$\mu(s) = s \left(\log\left(\frac{1}{s}\right) + 1 \right)$	& Log-Lip-continuity\\[4pt]
			$\mu(s) = s \left(\log\left(\frac{1}{s}\right) + 1\right)\log^{[m]}\left(\frac{1}{s}\right)$ 	& Log-Log$^{[m]}$-Lip-continuity\\[5pt]
			$\mu(s) = s^\alpha,\,\alpha\in(0,\,1)$	& H{\"o}lder-continuity \\[5pt]
			$\mu(s) = \left(\log\left(\frac{1}{s}\right) + 1 \right)^{-\alpha},\,\alpha\in(0,\,\infty)$	& Log$^{-\alpha}$-continuity\\[5pt]
			\bottomrule
		\end{tabulary}
	\caption{Some examples of moduli of continuity and how they are commonly referred to.}
	\label{table1}
	\end{table}

	Following \cite{Cicognani.2017}, we also use weight sequences $\{K_p\}_p$ to describe the regularity of the coefficients in space.
	\begin{Definition}
		Let $\{K_p\}_p$ be a positive, increasing sequence of real numbers. We define the space $\B_K^\infty = \B_K^\infty(\R^n)$ by
		\begin{equation*}
		\B_K^\infty(\R^n) = \Big\{ f \in C^\infty(\R^n)\,|\, \sup_{x\in\R^n} |D_x^\beta f(x)| \leq C K_{|\beta|}\, \text{for all } \beta \in \N^n \Big\}.
		\end{equation*}
		
		By $\B^\infty=\B^\infty(\R^n)$ we denote the space of all smooth functions that have bounded derivatives.
	\end{Definition}
	
	Let us now turn to weakly hyperbolic equations, where it is well-known that in general there is no $C^\infty$ well-posedness but one has to work in Gevrey spaces $G^s$, where the order $s$ depends on the order of the equation or the maximal multiplicity of the characteristic roots. For an equation of order $m$ one has to assume that $1 \leq s < \frac{m}{m-1}$; or if $\kappa$ is the maximal multiplicity of the characteristic roots, one has to assume $1 \leq s < \frac{\kappa}{\kappa-1}$, to obtain well-posedness in $G^s$ (see e.g. \cite{Steinberg.1975, Bronshtein.1976, Kajitani.1983}).
	
	One way to increase the upper bound on $s$ is to pose Levi conditions, that link the coefficients of the lower order terms to the coefficients of the principal part.
	For example, in \cite{Ivrii.1976} the author considered the Cauchy problem for the operator
		\begin{equation*}
		L = \partial_t^2 - t^{2l} \partial_x^2 - t^k \partial_x,
		\end{equation*}
		and proved well-posedness in $C^\infty(\R)$ for $k \geq l-1$, and in $G^s(\R)$ for $k < l-1$ if $1 \leq s < \frac{2l-k}{l-1-k}$.
		
	A more general approach to Levi conditions is the use of shape functions, which describe the speed at which characteristics coincide.
	We introduce shape functions following \cite{Ishida.2002, Yagdjian.1996, Yagdjian.1997}.
		\begin{Definition}[Shape functions]\label{SF:DEF}
		Let $\lambda \in C^\infty([0,\,T])$ be such that $\lambda(0) = \lambda^\prime(0) = 0$ and $\lambda(t), \,\lambda^\prime(t) > 0$, whenever $t \neq 0$. For $\lambda(t)$ we define $\Lambda(t) = \int_0^t \lambda(r) \rmd r$ and assume that
		\begin{align}
		\lambda^m\Lambda^{1-m} \in C^\infty([0,\,T])&,\\
		c_0  \frac{\lambda(t)}{\Lambda(t)} \leq  \frac{\lambda^\prime(t)}{\lambda(t)} \leq c   \frac{\lambda(t)}{\Lambda(t)}&, \text{ for all } t \in (0,\,T],\,c_0 > \frac{s(m-1)}{(s-1)m},\\
		|\lambda^{(k)}(t)| \leq c  \bigg( \frac{\lambda^\prime(t)}{\lambda(t)}\bigg)^{k-1} |\lambda^\prime(t)|&, \text{ for all } t \in (0,\,T], \, k =1,\,2,\,\ldots,
		\end{align}
		where $m$ is the order of the weakly hyperbolic Cauchy problem of interest and $s \geq \frac{m}{m-1}$ is fixed.
		
		Typical examples of shape functions are
		\begin{equation*}
		\lambda(t) = t^l,\,l > m-1,\qquad \lambda(t) = \exp(-|t|^{-r}),\,r > 0.
		\end{equation*}
	\end{Definition}
	We use shape functions to describe the behavior of the coefficients of the principal part and thus the behavior of the characteristic roots. The Levi condition is then formulated in terms of the shape function.
	For example, in \cite{Ishida.2002} the authors considered an operator
	\begin{equation*}
		L = \partial_t^2 - \underbrace{\sum\limits_{j,\,k = 1}^n a_{j,\,k}(t) \partial^2_{x_j x_k}}_{=a_2(t,\,\partial_x)}- \underbrace{\sum\limits_{j=1}^n a_j(t) \partial_{x_j}}_{=a_1(t,\,\partial_x)},	
	\end{equation*}
	with coefficients in $C^1([0,\,T])$. They assumed that
	\begin{enumerate}[label = (\roman*)]
		\item $a_2(t,\,\xi) \sim |\xi|^2 \lambda(t)^2$, $(t,\,\xi) \in [0,\,T]\times\R^n$,
		\item $|\partial_t a_2(t,\,\xi)| \lesssim \lambda(t)^3 \Lambda(t)^{-\frac{s}{s-1}} |\xi|^2$, $(t,\,\xi) \in (0,\,T]\times\R^n$,
		\item $|\partial_t^k a_j(t)| \lesssim \lambda(t)^{2+k} \Lambda(t)^{-\frac{s}{s-1}\,(1+k)}$ for $k = 0,1$ and $t\in(0,\,T]$,
	\end{enumerate}
	and were able to prove well-posedness of the Cauchy problem in the weighted spaces
	\begin{equation*}
	L^2_{s,\,\rho} = \Big\{u \in L^2(\R^n)\,|\,\exp(\rho \jbl D_x\jbr^{\frac{1}{s}}) u(x) \in L^2(\R^n)\Big\}.
	\end{equation*}
	
	Shape functions may also be used when treating weakly hyperbolic Cauchy problems with oscillations (see e.g. \cite{Reissig.1999}) or even to consider problems where the characteristics coincide with different speeds (see e.g. \cite{Kajitani.2002}).
	
	In our approach, we use shape functions and propose a generalized Levi condition. Instead of a typical Gevrey Levi condition (see e.g. \cite{Yagdjian.1996, Yagdjian.1997}), where we would assume that the coefficients of the lower order terms satisfy
	\begin{equation*}
	|b_{m-j,\,\gamma}(t,\,x)| \lesssim \lambda(t)^{m-j} \bigg(\frac{1}{\Lambda(t)^{\frac{s}{s-1}}}\bigg)^{m-j-|\gamma|},
	\end{equation*}
	we assume that
	\begin{equation}\label{INTRO:Levi}
	| b_{m-j,\,\gamma}(t,\,x)| \lesssim\lambda(t)^{m-j} (w(\Lambda(t)))^{m(m-j-|\gamma|)},
	\end{equation}
	where
	\begin{equation*}
	(w(\Lambda(t)))^m =  \Lambda(t)^{-\frac{s}{s-1}} (\log^{[\widetilde m]}(\Lambda(t)^{-1}))^{\widetilde \beta},
	\end{equation*}
	with $s \geq \frac{m}{m-1}$, $\widetilde m \in \N$ and $\widetilde\beta \in \R$.
	We note that this special choice of $w(\Lambda(t))$ enables us to work in scales of Gevrey-type Levi conditions.
	
	In this paper, we prove a sufficient well-posedness result for a class of weakly hyperbolic Cauchy problems, where the coefficients of the principal part are $\mu$-continuous in time and belong to $B^\infty_K$ in space and the coefficients of the lower order terms satisfy the generalized Levi condition \eqref{INTRO:Levi}. Our result states that the effects of the degeneracy and of the low regularity are independent of each other and that possible solution spaces are dictated by the dominant influence only.
	
	The paper is organized as follows: Section~\ref{RESULTS} states the main results of this paper. Examples and remarks are discussed in Section~\ref{EXAMPLES}. Section~\ref{DEF} reviews some definitions and provides an introduction to the pseudodifferential calculus used in this paper. Finally, in Section~\ref{PROOF} we proceed to prove the theorem of Section~\ref{RESULTS}.
	
	\section{Statement of the results}\label{RESULTS}

	For the formulation of the theorem, it is helpful to introduce the following notation.
	We introduce the functions $w(\Lambda(t))$ and $W(\Lambda(t))$ by
	\begin{equation}\label{CP2:DefW}
	\begin{aligned}[m]
	(w(\Lambda(t)))^m &=  \Lambda(t)^{-\frac{s}{s-1}} \big(\log^{[\widetilde m]}(\Lambda(t)^{-1})\big)^{\widetilde \beta},\\
	\partial_t W(\Lambda(t)) &= \lambda(t) w(\Lambda(t)),
	\end{aligned}
	\end{equation}
	for $s \geq \frac{m}{m-1}$, $\widetilde m \in \N$ and $\widetilde \beta \in \R$.
	Furthermore, by $t_\xi = t(\xi,\,N)$ we denote the positive solution to
	\begin{equation}\label{CP2:DefZones}
	N 	\big(w(\Lambda(t_\xi))\big)^m = \jxi,\,\quad N > 0.
	\end{equation}

	Let us consider the Cauchy problem
	\begin{equation}\label{CP2:CP}
	\begin{aligned}
	D_t^m u = \sum\limits_{j=0}^{m-1} \sum\limits_{|\gamma|+j = m} \lambda(t)^{m-j}a_{m-j,\,\gamma}(t,\,x) D_x^\gamma D_t^j u\\
	+ \sum\limits_{|\gamma|+j < m} b_{m-j,\,\gamma}(t,\,x) D_x^\gamma D_t^j u,\\
	D_t^{k-1}u(0,\,x) = g_k(x),\,k = 1,\,\ldots,\,m,\,(t,\,x) \in [0,\,T]\times\R^n,
	\end{aligned}
	\end{equation}
	under the following conditions:
	\begin{enumerate}[label = (A\arabic*),align = left, leftmargin=*]
		\item \label{CP2:ShapeFunc} The function $\lambda(t)$ is a shape function (see Definition~\ref{SF:DEF}).
		
		\item\label{CP2:Hyperbolic} For $\lambda(t) \equiv 1$, the Cauchy problem \eqref{CP2:CP} would be strictly hyperbolic.
		
		\item \label{CP2:CoeffPrinc} The coefficients of the principal part $a_{m-j,\,\gamma}=a_{m-j,\,\gamma}(t,\,x)$ belong to $C^\mu\big([0,\,T];\,B^\infty_K\big)$ and satisfy
		\begin{equation*}
		\big|D_x^\beta a_{m-j,\,\gamma}(t,\,x) - D_x^\beta a_{m-j,\,\gamma}(s,\,x)\big| \leq C K_{|\beta|} \mu(|t-s|),
		\end{equation*}
		for $t,\,s\in [0,\,T],\,0 \leq |t-s| \leq 1, \,x\in\R^n$.
		
		\item \label{CP2:LeviCond} The coefficients of the lower order terms $b_{m-j,\,\gamma}=b_{m-j,\,\gamma}(t,\,x)$ belong to $C\big([0,\,T];\,B^\infty_K\big)$ and
		\begin{equation*}
		\big|D_x^\beta b_{m-j,\,\gamma}(t,\,x)\big| \leq C K_{|\beta|} \lambda(t)^{m-j} \big(w(\Lambda(t))\big)^{m(m-j-|\gamma|)},
		\end{equation*}
		for all $(t,\,x) \in (0,\,T]\times\R^n$.
				
		\item\label{CP2:Data} The initial data $g_k \in H^{\nu+m-k}_{\eta,\,\delta_1}$, where $\nu \in \R$, $\delta_1 > 0$.
		
		\item \label{CP2:MOC} The modulus of continuity $\mu = \mu(s)$ can be written in the form
		\begin{equation*}
		\mu(s) = s\vp(s^{-1}),
		\end{equation*}
		where $\vp = \vp(s)$ is a non-decreasing, smooth function on $[c,\,+\infty),\,c > 0$.
		
		\item \label{CP2:EtaAndK} The weight function $\eta = \eta(\jxi)$ and the sequence of constants $\{K_p\}_p$ satisfy the relation
		\begin{equation*}
		\inf\limits_{p\in\N}\frac{K_p}{\jxi^{p}} \leq C e^{-\delta_0 \eta(\jbl\xi\jbr)},
		\end{equation*}
		for large $|\xi|$ and some $\delta_0 > 0$.
		
		\item \label{CP2:EstEta} The functions
		\begin{equation*}\eta = \eta(\jxi)  \text{ and } M(\jxi) = W(\Lambda(t_\xi))(w(\Lambda(t_\xi)))^{m-1} + \vp(\jxi),
		\end{equation*}
		are smooth and satisfy
		\begin{equation}\label{SH:MR:THEOREM:WEAK:ASSUME:Eta:diff}
		\bigg|\frac{\rmd^k}{\rmd s^k} M (s) \bigg| \leq C_k s^{-k} M(s),
		\end{equation}
		for all $k \in \N$ and large $s \in \R^+$ and
		\begin{equation}\label{CP2:EstEta:subadd}
		\eta(\jbl \xi + \zeta \jbr) \leq \eta(\jxi) + \eta(\jbl \zeta \jbr),\qquad
		M(\jbl \xi + \zeta \jbr)\leq M(\jxi) + M(\jbl \zeta \jbr),
		\end{equation}
		for all large $|\xi|,\,|\zeta|,\,\xi,\,\zeta \in \R^n$.
		
		\item \label{CP2:WAtZero} We have
		\begin{equation*}
			\lim\limits_{t\rightarrow 0+} \lambda(t)^m \big(w(\Lambda(t))\big)^{m(m-1)} = 0.
		\end{equation*}
				
		\item \label{CP2:GlobalResult}
		The weight function $\eta = \eta(\jxi)$ satisfies
		\begin{equation}\label{CP2:DefEta}
		W(\Lambda(t_\xi))\big(w(\Lambda(t_\xi))\big)^{m-1} + \vp(\jxi) = o(\eta(\jxi)).
		\end{equation}
		
		\item \label{CP2:LocalResult}
		The weight function $\eta = \eta(\jxi)$ satisfies
		\begin{equation}\label{CP2:DefEtaO}
		W(\Lambda(t_\xi))\big(w(\Lambda(t_\xi))\big)^{m-1} + \vp(\jxi) = O(\eta(\jxi)).
		\end{equation}
		
	\end{enumerate}	
	
	\begin{Theorem}\label{CP2:Theorem}
		Consider the \hyperref[CP2:CP]{Cauchy problem \eqref{CP2:CP}}. Assume the conditions \ref{CP2:ShapeFunc}-\ref{CP2:WAtZero} and \ref{CP2:GlobalResult}, then we have global (in time) well-posedness, i.e. there is a global (in time) solution
		\begin{equation*}
		u \in \bigcap\limits_{j = 0}^{m-1} C^{m-1-j}\big([0,T];\, H^{\nu+j}_{\eta,\,\delta}\big),
		\end{equation*}
		where $\delta < \min\{\delta_0,\,\delta_1\}$.
		
		If we assume \ref{CP2:LocalResult} instead of \ref{CP2:GlobalResult}, we only have local (in time) well-posedness, i.e. there is a local (in time) solution
		\begin{equation*}
		u \in \bigcap\limits_{j = 0}^{m-1} C^{m-1-j}\big([0,T^\ast];\, H^{\nu+j}_{\eta,\,\delta}\big),
		\end{equation*}
		with $0<T^\ast \leq T$.
	\end{Theorem}

	\section{Examples and remarks}\label{EXAMPLES}
	
	Let us begin with some remarks about the previous theorem.
	\begin{Remark}
		In the definition of $w(\Lambda(t))$ in \eqref{CP2:DefW} we have the requirement $s \geq \frac{m}{m-1}$. This bound is due to Steinberg's result~\cite{Steinberg.1975}, that we have Gevrey well-posedness in $G^s$ if $ 1\leq s < \frac{m}{m-1}$ if the coefficients are smooth, without a Levi condition.
	\end{Remark}

	\begin{Remark}
		Assumption \eqref{CP2:DefW} limits our choice of admissible Levi conditions to scales of Gevrey type Levi conditions. However, it is also possible to work with more general functions $w(\Lambda(t))$ as long as they satisfy all relations of Proposition~\ref{CP2:Remark:EstLambdaW}.
	\end{Remark}	

	\begin{Remark}
		The division of the extended phase space into two zones, governed by the separating line given in \eqref{CP2:DefZones}, is done in such a way, that the loss of derivatives due to the weak hyperbolicity is the same in each zone.	This generalized definition of the zones, given by assumption \eqref{CP2:DefZones}, is compatible with the well-known definition of the zones for Gevrey type Levi conditions (see e.g. \cite{Ishida.2002, Yagdjian.1996, Yagdjian.1997}).
	\end{Remark}

	\begin{Remark}[\cite{Cicognani.2017}]
	Assumption \ref{CP2:EtaAndK} describes the connection between the weight function $\eta$ of the solution space and the behavior of the coefficients with respect to the spatial variables. In a way, we may interpret this condition as a multiplication condition in the sense that the regularity of the coefficients in $x$ has to be such that the product of coefficients and the solution stays in the solution space. This means that the weight sequence $\{K_p\}_p$ and the weight function $\eta$ have to be compatible in a certain sense. One way to ensure that they are compatible is to choose them such that the function space of all functions $f \in C^\infty(\R^n)$ with
	\begin{equation*}
	\sup\limits_{x\,\in \R^n}\big|D_x^\alpha f(x)\big| \leq C K_{|\alpha|},
	\end{equation*}
	and the function space of all functions $f \in L^2(\R^n)$ with
	\begin{equation*}
	e^{\eta(\langle D_x \rangle)} f \in L^2(\R^n)
	\end{equation*}
	coincide. For results concerning the conditions on $\eta$ and $\{K_p\}_p$ under which both spaces coincide, we refer the reader to \cite{Bonet.2007, Pascu.2010, Reich.2016}.
	\end{Remark}
	
	\begin{Remark}[\cite{Cicognani.2017}]
	Assumption \ref{CP2:EstEta} provides some relations that are used in the pseudodifferential calculus. Condition \eqref{SH:MR:THEOREM:WEAK:ASSUME:Eta:diff} for $\eta$ and $M$ is not really a restriction. First of all, the first summand of $M$ satisfies this relation anyway. If $\eta$ or $\vp$ happen to be not smooth, we can define equivalent weight functions, that are smooth and satisfy \eqref{SH:MR:THEOREM:WEAK:ASSUME:Eta:diff}.
	
	The difficulty of checking whether condition \eqref{CP2:EstEta:subadd}, is satisfied, certainly depends on the choice of $\eta$. For $M$ this condition is easily verified due to our special choice of $w(\Lambda(t))$.  However, in some cases it may be easier to verify that $\eta$  and $M$ belong to a certain class of weights, for which \eqref{CP2:EstEta:subadd} is satisfied. An example of such a class is introduced in Definition~3.7 in \cite{Reich.2016}.
	\end{Remark}
	
	\begin{Remark}
		Assumption \ref{CP2:WAtZero} limits the cases we can treat with this approach. For a Gevrey type Levi condition (i.e. $(w(\Lambda(t)))^m = \Lambda(t)^{-\frac{s}{s-1}}$) and the choice of $\lambda(t) = t^l$, this condition implies that $s > m$, which means that our approach is only applicable if the degeneracy is sufficiently strong. We note that a condition like this is also present in other results (see e.g. \cite{Ishida.2002, Yagdjian.1996, Yagdjian.1997}), although it is hidden in the assumptions on the admissible shape functions.
	\end{Remark}
	
	\begin{Remark}
		The crucial condition which determines the spaces in which we have well-posedness is condition \eqref{CP2:DefEta}. In this condition we see, that each effect produces a weight. The term $W(\Lambda(t_\xi)) (w(\Lambda(t_\xi)))^{m-1}$ is due to the weak hyperbolicity, whereas the term $\vp(\jxi)$ is due to the low regularity of the coefficients. Most importantly, we see that the weights coming from each effect are added up, which means that they are independent and do not influence each other. Furthermore, it is clear that we only feel the effect of the stronger weight. This behavior can be seen more clearly in the following examples.
	\end{Remark}
	
	In the following, we compute some examples. In each example, we first choose a Levi condition and compute the weight related to the particular choice of $w(\Lambda(t))$. Then, we look at possible choices of moduli of continuity $\mu$ to describe the regularity of the coefficients in time. Depending on this modulus of continuity and the Levi condition, we look for a suitable weight function $\eta$ which satisfies \eqref{CP2:DefEta}.
	Having chosen $\eta$, we specify the regularity of the coefficients in space by choosing a sequence of constants $\{K_p\}_p$ such that \ref{CP2:EtaAndK} is satisfied.
	
	Let us begin with some examples of typical Gevrey type Levi conditions.
	
	\begin{Example}
		Let $(w(\Lambda(t)))^m = (\Lambda(t))^{- \frac{s}{s-1}}$ resulting in $W(\Lambda(t)) = \Lambda(t)^{1-\frac{s}{m(s-1)}}$. The definition of the zones yields that $\Lambda(t_\xi) \sim \jxi^{-\frac{s-1}{s}}$, which gives that
		\begin{equation*}
			W(\Lambda(t_\xi))(w(\Lambda(t_\xi)))^{m-1} \sim \jxi^{\frac{1}{s}}.
		\end{equation*}
		As expected, the Gevrey type Levi condition yields a Gevrey type weight.
		
		As for the choice of a modulus of continuity, we can see from Table~\ref{table2} that we may allow the coefficients to be H{\"o}lder continuous of order $\alpha = 1-\frac{1}{s}$ or smoother, without changing the overall weight of our solution space. In these cases, we may also use the well-known inequality
		\begin{equation*}
		\inf\limits_{p \in \N} (p!)^{s^\ast} (A \jxi^{-1})^p \leq C e^{-\delta_0 \jxi^{\frac{1}{s^\ast}}},
		\end{equation*}
		to obtain that a possible weight sequence $\{K_p\}_p$ is the Gevrey weight sequence $K_p = (p!)^{s^\ast} A^{p}$. In that way assumption~\ref{CP2:EtaAndK} is satisfied. Also assumption~\ref{CP2:EstEta} can be easily verified for this choice of $w(\Lambda(t))$ and $\eta$.
				
		In all these cases, we have well-posedness in Gevrey type spaces
		\begin{equation*}
			H^\nu_{\eta,\,\delta}(\R^n) = \Big\{f \in \Sw^\prime(\R^n)\,|\, \jbl D_x\jbr^\nu e^{\delta\eta(\jbl D_x \jbr)} f(x) \in L^2(\R^n)\Big\},
		\end{equation*}
		with $\eta(\jxi) = \jxi^{\frac{1}{s^\ast}}$ and $1 \leq s^\ast < s$, which is a well-known result (see e.g. \cite{Ishida.2002, Yagdjian.1996, Yagdjian.1997}).
		
		However, if we choose $\mu(s) =  \left(\log\left(\frac{1}{s}\right) + 1 \right)^{-\alpha},\quad \alpha\in(0,\,\infty)$, the weight $\vp(\jxi) = \jxi\log(\jxi))^{-\alpha} $ clearly dominates $\jxi^{\frac{1}{s}}$ and we have to choose
		\begin{equation*}
		\eta(\jxi) =  \jxi (\log(\jxi))^{-\kappa},
		\end{equation*}
		where $0 < \kappa < \alpha$.
		In view of Definition~9 and Example~25 in \cite{Bonet.2007}, we find that condition \ref{CP2:EtaAndK} is satisfied if we choose
		\begin{equation*}
		K_p = ((p+1)(\log(e+p)))^p.
		\end{equation*}
		
		Checking that assumption~\eqref{CP2:EstEta:subadd} is satisfied, may not seem straightforward. However, we can easily check that $\eta$ and $\vp$ belong to the set $\W(\R)$ which was introduced in \cite{Reich.2016}. In \cite{Reich.2016} the author proves that all functions in  $\W(\R)$ satisfy an even stronger condition than \eqref{CP2:EstEta:subadd}.
		
		In that way all assumptions are satisfied and we have well-posedness in spaces
		\begin{equation*}
		H^\nu_{\eta,\,\delta}(\R^n) = \Big\{f \in \Sw^\prime(\R^n)\,|\, \jbl D_x\jbr^\nu e^{\delta\eta(\jbl D_x \jbr)} f(x) \in L^2(\R^n)\Big\},
		\end{equation*}
		with $\eta(\jxi) =  \jxi (\log(\jxi))^{-\kappa}$ and $0 < \kappa < \alpha$. We note that these spaces are very close to the space of analytic functions even though we assumed a Gevrey type Levi condition.
	\end{Example}

	\begin{table}[htbp]
	\begin{tabulary}{\textwidth}{LL}
		\toprule
		modulus of continuity &generated weight\\
		\midrule
		$\mu(s) = s$	& $\vp(\jxi) = 1$\\[5pt]
		$\mu(s) = s \left(\log\left(\frac{1}{s}\right) + 1 \right)$	& $\vp(\jxi) = \log(\jxi)$\\[4pt]
		$\mu(s) = s \left(\log\left(\frac{1}{s}\right) + 1\right)\log^{[\widetilde n]}\left(\frac{1}{s}\right)$ 	& $\vp(\jxi) = \log(\jxi) \log^{[\widetilde n]}(\jxi)$ \\[5pt]
		$\mu(s) = s^\alpha,\, \alpha\in(0,\,1)$	& $\vp(\jxi) = \jxi^{1-\alpha}$ \\[5pt]
		$\mu(s) = \left(\log\left(\frac{1}{s}\right) + 1 \right)^{-\alpha},\, \alpha\in(0,\,\infty)$	& $\vp(\jxi) = \jxi\log(\jxi))^{-\alpha} $\\[5pt]
		\bottomrule
	\end{tabulary}
	\caption{Moduli of continuity $\mu$ and the respective, generated weights $\vp$.}
	\label{table2}
	\end{table}

	Next, let us consider a Levi condition which is a little less restrictive than the usual Gevrey type Levi condition.
	
	\begin{Example}
		Let
		\begin{equation*}
		(w(\Lambda(t)))^m = (\Lambda(t))^{- \frac{s}{s-1}} \log(\Lambda(t)^{-1}),
		\end{equation*}
		which gives
		\begin{equation*}
		W(\Lambda(t)) \sim \Lambda(t)^{1-\frac{s}{m(s-1)}}\log(\Lambda(t)^{-1})^{\frac{1}{m}},
		\end{equation*}
		for small $t$ i.e. large $\Lambda(t)^{-1}$.
		The definition of the zones yield that
		\begin{equation*}
		\Lambda(t_\xi) \sim \jxi^{-\frac{s-1}{s}} (\log(\jxi^{\frac{s-1}{s}}))^{\frac{s-1}{s}},
		\end{equation*}
		for large $|\xi|$, which gives that
		\begin{equation}\label{Example:Weight1}
		W(\Lambda(t_\xi))(w(\Lambda(t_\xi)))^{m-1} \sim \frac{s-1}{s}\bigg(\frac{\jxi}{\frac{s-1}{s}\log(\jxi)}\bigg)^{\frac{1}{s}}\log\bigg(\frac{\jxi}{\frac{s-1}{s}\log(\jxi)}\bigg).
		\end{equation}
		We see that this Levi condition, that is a little less restrictive than the usual Gevrey type Levi condition leads to a weight, that is very close to a Gevrey weight. However, due to the $\log$-terms we are slightly below the Gevrey weight $\jxi^{\frac{1}{s}}$.

		Again, looking at Table~\ref{table2}, we can see that this weight dominates the weights due to the modulus of continuity until we assume H{\"o}lder continuous coefficients or worse. If the coefficients are worse than H{\"o}lder, clearly their weight is dominant. If the coefficients are H{\"o}lder continuous of order $\alpha$, the situation is more delicate. If $\alpha < 1- \frac{1}{s}$, the weight given by \eqref{Example:Weight1} is dominant; if $\alpha \geq  1- \frac{1}{s}$ the weight $\jxi^{1-\alpha}$ is dominant.
		
		In the latter case (i.e. $\alpha \geq 1-\frac{1}{s}$), it is clear that we are again working in Gevrey spaces and that we can employ the same weight sequence $\{K_p\}_p$ and weight function $\eta$ as in the previous example. In these cases, we have well-posedness in Gevrey type spaces.
		
		If $\alpha < 1- \frac{1}{s}$, the weight given by \eqref{Example:Weight1} dominates and we may choose any weight function $\eta$ that grows faster than this weight. One example would be
		\begin{equation*}
		\eta (\jxi) = \bigg(\frac{\jxi}{\frac{s-1}{s}\log(\jxi)}\bigg)^{\frac{1}{s}}\log\bigg(\frac{\jxi}{\frac{s-1}{s}\log(\jxi)}\bigg) \log\bigg(\log\bigg(\frac{\jxi}{\frac{s-1}{s}\log(\jxi)}\bigg)\bigg).
		\end{equation*}
		However, for this particular choice of $\eta$ it is quite challenging to find an ``optimal'' weight sequence $\{K_p\}_p$. A simple solution to that problem is to just use the Gevrey weight sequence $K_p = (p!)^{s^\ast} A^{p}$ again. With this weight sequence, it is clear that assumption~\ref{CP2:EtaAndK} is satisfied. For more detailed considerations about finding and choosing a weight sequence in this setting, we refer the reader to the example with Log-Log$^{[m]}$-Lip continuous coefficients in \cite{Cicognani.2017}.
		In these cases, we have well-posedness in spaces that are very close but a little bit smaller than the classical Gevrey space $G^{s}$.
		
		Again, if we choose $\mu(s) =  \left(\log\left(\frac{1}{s}\right) + 1 \right)^{-\alpha},\quad \alpha\in(0,\,\infty)$, the weight $\vp(\jxi) = \jxi\log(\jxi))^{-\alpha} $ clearly dominates the weight given by \eqref{Example:Weight1} and we have to choose
		\begin{equation*}
		\eta(\jxi) =  \jxi (\log(\jxi))^{-\kappa},
		\end{equation*}
		where $0 < \kappa < \alpha$.
		Again, choosing
		\begin{equation*}
		K_p = \big((p+1)(\log(e+p))\big)^p,
		\end{equation*}
		ensures that condition \ref{CP2:EtaAndK} is satisfied.
				
		In that way all assumptions are satisfied and we have well-posedness in spaces
		\begin{equation*}
		H^\nu_{\eta,\,\delta}(\R^n) = \Big\{f \in \Sw^\prime(\R^n)\,|\, \jbl D_x\jbr^\nu e^{\delta\eta(\jbl D_x \jbr)} f(x) \in L^2(\R^n)\Big\},
		\end{equation*}
		with $\eta(\jxi) =  \jxi (\log(\jxi))^{-\kappa}$ and $0 < \kappa < \alpha$. We note that these spaces are very close to the space of analytic functions even though we assumed a Gevrey type Levi condition.
	\end{Example}
	
	\begin{Remark}
		Due to the special choice of the function $w(\Lambda(t))$, it is possible to provide a general characterization of the weights $W(\Lambda(t_\xi)) (w(\Lambda(t_\xi)))^{m-1}$. For general
		\begin{equation*}
		(w(\Lambda(t)))^m = \Lambda(t)^{-\frac{s}{s-1}} \big(\log^{[\widetilde m]}(\Lambda(t)^{-1})\big)^{\widetilde \beta},
		\end{equation*}
		with $ \widetilde \beta \neq 0$, we have
		\begin{equation*}
		W(\Lambda(t)) \sim \Lambda(t)^{1-\frac{s}{m(s-1)}} \big(\log^{[\widetilde m]}(\Lambda(t)^{-1})\big)^{\frac{\widetilde \beta}{m}},
		\end{equation*}
		for small $t > 0$. We obtain that
		\begin{equation*}
		W(\Lambda(t_\xi)(w(\Lambda(t_\xi)))^{m-1} \sim \Lambda(t_\xi)^{-\frac{1}{s-1}} \log^{[\widetilde m]}(\Lambda(t_\xi)^{-1})^{\widetilde \beta},
		\end{equation*}
		for small $t_\xi$, i.e. large $|\xi|$.
		The general definition of the zones gives
		\begin{equation*}
			\Lambda(t_\xi) \sim \jxi^{-\frac{s-1}{s}} \big(\log^{[\widetilde m]}(\jxi^{\frac{s-1}{s}})\big)^{\frac{s-1}{\widetilde \beta s}},
		\end{equation*}
		for small $t_\xi$, i.e. large $|\xi|$, which allows us to conclude that
		\begin{equation*}
		\begin{aligned}[t]
		W(\Lambda(t_\xi)(w(\Lambda(t_\xi)))^{m-1} &\sim \Bigg(\frac{\jxi}{(\log^{[\widetilde m]}(\jxi^{\frac{s-1}{s}}))^{\frac{1}{\widetilde \beta}}}\Bigg)^{\frac{1}{s}}\\
		&\times \Bigg(\log^{[\widetilde m]}\Bigg(\frac{\jxi}{(\log^{[\widetilde m]}(\jxi^{\frac{s-1}{s}}))^{\frac{1}{\widetilde \beta }}}\Bigg)^{\frac{s-1}{s}}\Bigg)^{\widetilde \beta},
		\end{aligned}
		\end{equation*}
		for large $|\xi|$. Depending on the sign of $\widetilde \beta$, this means that this weight always suggests working slightly below or slightly above the related Gevrey space $G^s$.
	\end{Remark}

	\section{Definitions and tools}\label{DEF}
	
	We begin by reviewing some notations.
	
	Let $x = (x_1,\,\,\ldots,\,x_n)$ be the variables in the $n$-dimensional Euclidean space $\R^n$ and by $\xi = (\xi_1,\,\ldots,\,\xi_n)$ we denote the dual variables. Furthermore, we set $\jbl\xi\jbr^2 = 1 + |\xi|^2$.
	We use the standard multi-index notation. Precisely, let $\Z$ be the set of all integers and $\Z_+$ the set of all non-negative integers. Then $\Z^n_+$ is the set of all $n$-tuples $\alpha = (\alpha_1,\,\ldots,\,\alpha_n)$ with $a_k \in \Z_+$ for each $k = 1,\,\ldots,\,n$. The length of $\alpha \in \Z^n_+$ is given by $|\alpha| = \alpha_1 + \ldots + \alpha_n$.\\
	Let $u = u(t,\,x)$ be a differentiable function, we then write
	\begin{equation*}
		u_t(t,\,x) = \partial_t u (t,\,x) = \frac{\partial}{\partial t} u(t,\,x),
	\end{equation*}
	and
	\begin{equation*}
		\partial_x^\alpha u (t,\,x) = \left(\frac{\partial}{\partial x_1}\right)^{\alpha_1} \ldots\left(\frac{\partial}{\partial x_n}\right)^{\alpha_n} u(t,\,x).
	\end{equation*}
	Using the notation $D_{x_j} = -\I \frac{\partial}{\partial x_j}$, where $\I$ is the imaginary unit, we write also
	\begin{equation*}
		D_x^\alpha = D_{x_1}^{\alpha_1} \cdots D_{x_n}^{\alpha_n}.
	\end{equation*}
	Similarly, for $x\in \R^n$ we set
	\begin{equation*}
		x^\alpha = x_1^{\alpha_1} \cdots x_n^{\alpha_n}.
	\end{equation*}
	In the context of pseudodifferential operators and the related symbolic calculus, we sometimes use the notation
	\begin{equation*}
		a^{(\alpha)}_{(\beta)}(x,\,\xi) = \partial_\xi^\alpha D_x^\beta a(x,\,\xi).
	\end{equation*}
	
	Let $f$ be a continuous function in an open set $\Omega \subset \R^n$. By $\supp f$ we denote the support of $f$, i.e. the closure in $\Omega$ of $\{x \in \Omega\,|\,f(x) \neq 0\}$. By $C^k(\Omega)$, $0 \leq k \leq \infty$, we denote the set of all functions $f$ defined on $\Omega$, whose derivatives $\partial^\alpha_x f$ exist and are continuous for $|\alpha| \leq k$. By $\Czi(\Omega)$ we denote the set of all functions $f \in C^\infty(\Omega)$ having compact support in $\Omega$. The Sobolev space $H^{k,p}(\Omega)$ consists of all functions that are $k$ times differentiable in Sobolev sense and have (all) derivatives in $L^p(\Omega)$.
	
	For two functions $f=f(x)$ and $g=g(x)$ we write
	\begin{align*}
		f(x) = o(g(x))\qquad \text{ if }\qquad \lim\limits_{|x|\rightarrow \infty} \frac{f(x)}{g(x)} = 0,
	\end{align*}
	and we use the notation
	\begin{equation*}
		f(x) = O(g(x))\qquad \text{ if }\qquad \limsup\limits_{|x|\rightarrow \infty} \frac{f(x)}{g(x)} \leq C.
	\end{equation*}
	We use $C$ as a generic positive constant which may be different even in the same line.
	
	An import tool in our approach is the division of the extended phase space into zones. Depending on the Levi condition, we divide the phase space into two zones. For this purpose we define $t_\xi$ as a solution to the equation
	\begin{equation*}
	\jxi = N (w(\Lambda(t_\xi)))^m,
	\end{equation*}
	where $N$ is a positive constant. Using $t_\xi$ and the notation $J = [0,\,T]\times \R^n \times \R^n$ we define the so-called pseudodifferential zone
	\begin{equation}\label{Def:Zpd}
	\begin{aligned}
	\Zpd(N,\,M) &= \big\{(t,\,x,\,\xi) \in J\,|\, 0 \leq t \leq t_\xi,\,\jxi > M \big\}\\
	&= \big\{(t,\,x,\,\xi) \in J\,|\, \jxi \leq N (w(\Lambda(t)))^m,\,\jxi > M \big\},
	\end{aligned}
	\end{equation}
	and the so-called hyperbolic zone
	\begin{equation}\label{Def:Zhyp}
	\begin{aligned}
	\Zhyp(N,\,M) &= \big\{(t,\,x,\,\xi) \in J\,|\, t_\xi \leq t \leq T,\,\jxi > M \big\}\\
	&= \big\{(t,\,x,\,\xi) \in J\,|\, \jxi \geq N(w(\Lambda(t)))^m,\,\jxi > M \big\}.
	\end{aligned}
	\end{equation}
	
	Concerning the function $w(\Lambda(t))$, we state some relations that are obtained by straightforward computations thanks to the special structure of $w(\Lambda(t))$.
	
	\begin{Proposition}\label{CP2:Remark:EstLambdaW}
		Take $\lambda(t)$ and $\Lambda(t)$ as defined in \ref{CP2:ShapeFunc} and take $w(\Lambda(t))$ and $W(\Lambda(t))$ as given by \eqref{CP2:DefW}.
		\begin{enumerate}[label = (\roman*)]
			\item We have \begin{equation*}
			(w(\Lambda(t)))^m \geq C \frac{1}{\Lambda(t)},
			\end{equation*}
			if $0 < t$ is sufficiently small.
			
			\item For $t \in (0,\,T]$, we have
			\begin{equation*}
			0 < - \partial_t (w(\Lambda(t))) \leq \frac{\lambda(t)}{\Lambda(t)} \frac{1}{m} w(\Lambda(t)).
			\end{equation*}
			
			\item We have
			\begin{equation*}
			\int\limits_{0}^{t_\xi} \jxi^{\frac{1}{m}} \lambda(t) (w(\Lambda(t)))^{m-1} \rmd t \leq C W(\Lambda(t_\xi)) (w(\Lambda(t_\xi)))^{m-1},
			\end{equation*}
			and
			\begin{equation*}
			\int\limits_{t_\xi}^{T} \lambda(t) (w(\Lambda(t)))^{m} \rmd t \leq C W(\Lambda(t_\xi)) (w(\Lambda(t_\xi)))^{m-1}.
			\end{equation*}
		\end{enumerate}

	\end{Proposition}
	
	\subsection{Symbol classes and pseudodifferential calculus}
		
	We introduce standard symbol classes of pseudodifferential operators following \cite{Hormander.2007}.
	
	\begin{Definition}[$\Sy^m_{\rho,\,\delta}$ and $\OPS^{m}_{\rho,\,\delta}$]\label{APP:PSEUDO:DEF:SmRD}\label{APP:PSEUDO:DEF:Sm}
		Let $m,\,\rho,\,\delta$ be real numbers with $0 \leq \delta < \rho \leq 1$. Then we denote by $\Sy^m_{\rho,\,\delta} = \Sy^m_{\rho,\,\delta}(\R^n\times\R^n)$ the set of all $a\in C^\infty(\R^n \times \R^n)$ such that for  all multi-indexes $\alpha,\,\beta$ the estimate
		\begin{equation*}
		\big|D_x^\beta \partial_\xi^\alpha a(x,\,\xi)\big| \leq C_{\alpha,\,\beta} (1 + |\xi|)^{m - \rho|\alpha| + \delta |\beta|},
		\end{equation*}
		is valid for all $x,\,\xi \in \R^n$ and some constant $C_{\alpha,\,\beta}$.
		We write $\Sy^{-\infty}_{\rho,\,\delta} = \bigcap_m \Sy^m_{\rho,\,\delta}$, $\Sy^\infty_{\rho,\,\delta} = \bigcup_m \Sy^m_{\rho,\,\delta}$.
		For a given $a=a(x,\,\xi) \in \Sy^m_{\rho,\,\delta}$, we denote by $\OP(a) = a(x,\,D_x)$ the associated pseudodifferential operator, which is defined as
		\begin{align*}
		a(x,\,D_x) u(x) = \int\limits_{\R^n} e^{\I x\cdot \xi} a(x,\,\xi)  \hat{u}(\xi) \db \xi
		= \Osii e^{\I (x-y)\cdot \xi} a(x,\,\xi)  u(y) \rmd y \db\xi,
		\end{align*}
		where $\db \xi = (2\pi)^{-n} \rmd \xi$ and $\Osii$ means the oscillatory integral.\\
		By $\OPS^m_{\rho,\,\delta} = \OPS^m_{\rho,\,\delta}(\R^n)$ we denote the set of all pseudodifferential operators that are associated to some symbol in $\Sy^m_{\rho,\,\delta}$.
		Conversely, for $a \in \OPS^{m}_{\rho,\,\delta}$, we denote by $\sigma(a) \in \Sy^{m}_{\rho,\,\delta}$ the associated symbol.
	\end{Definition}
	We define symbol classes related to the zones we defined in \eqref{Def:Zpd} and \eqref{Def:Zhyp}. To describe the behavior of symbols in the pseudodifferential zone we introduce the function $\varrho = \varrho(t,\,\xi)$ to be a positive solution to
	\begin{equation*}
		\varrho^m(t,\,\xi) = 1 + \jxi \lambda^m(t) \big(w(\Lambda(t))\big)^{m(m-1)}.
	\end{equation*}
	
	\begin{Lemma}\label{Lemma:rho}
		The function $\varrho = \varrho(t,\xi)$ satisfies
		\begin{enumerate}[label = (\roman*),align = left, leftmargin=*]
			\item $0 \leq \partial_t \varrho(t,\xi)$,
			\item $ \frac{\partial_t \varrho(t,\xi)}{\varrho(t,\xi)} \leq \lambda(t)^{m} \frac{\lambda(t)}{\Lambda(t)} \big(w(\Lambda(t))\big)^{m(m-1)} \jxi$,
		\end{enumerate}
		for all $(t,x,\xi) \in[0,\,T]\times\R^n\times\R^n$.
	\end{Lemma}
	\begin{proof}
		For (i) we compute
		\begin{align*}
		\partial_t \varrho(t,\xi) &= \partial_t \Big(1+\jxi \lambda^m(t) (w(\Lambda(t)))^{m(m-1)}\Big)^{\frac{1}{m}}\\
		&= \begin{aligned}[t]
		\frac{1}{m} \varrho^{-(m-1)} \jxi &\Big(m \lambda^{m-1}(t) \lambda^\prime(t) (w(\Lambda(t)))^{m(m-1)}\\
		&+ m(m-1)\lambda^m(t) (w(\Lambda(t)))^{m(m-1)-1} \partial_t(w(\Lambda(t))) \Big).
		\end{aligned}
		\end{align*}
		We use Definition~\ref{SF:DEF} to estimate $\lambda^\prime(t) \geq c_0 \lambda(t) \frac{\lambda(t)}{\Lambda(t)}$ and (ii) from Proposition~\ref{CP2:Remark:EstLambdaW}. We obtain
		\begin{align*}
		\partial_t \varrho(t,\xi) &\geq \begin{aligned}[t]
		\frac{1}{m} \varrho^{-(m-1)} \jxi &\Big(c_0 m \lambda^{m}(t) \frac{\lambda(t)}{\Lambda(t)} (w(\Lambda(t)))^{m(m-1)}\\
		&- (m-1)\lambda^m(t) \frac{\lambda(t)}{\Lambda(t)} (w(\Lambda(t)))^{m(m-1)} \Big) \geq 0,
		\end{aligned}
		\end{align*}
		since $c_0 > \frac{m-1}{m}$. This proves the lower bound on $\varrho$.
		
		For (ii) we again use Definition~\ref{SF:DEF} to estimate $\lambda^\prime(t) \geq c \lambda(t) \frac{\lambda(t)}{\Lambda(t)}$ and (ii) from Proposition~\ref{CP2:Remark:EstLambdaW}. We have
		\begin{equation}\label{Lemma:rho:estDtRho}
		0 \leq \partial_t \varrho(t,\xi) \leq  C \varrho^{-(m-1)} \jxi  \lambda^{m}(t) \frac{\lambda(t)}{\Lambda(t)} \big(w(\Lambda(t))\big)^{m(m-1)}.
		\end{equation}
		From this we conclude
		\begin{align*}
			\frac{\partial_t \varrho(t,\xi)}{\varrho(t,\xi)} \leq C \frac{\jxi  \lambda^{m}(t) \frac{\lambda(t)}{\Lambda(t)} \big(w(\Lambda(t))\big)^{m(m-1)}}{\varrho(t,\,\xi)^m}\\\leq C\jxi  \lambda^{m}(t) \frac{\lambda(t)}{\Lambda(t)} (w(\Lambda(t)))^{m(m-1)},
		\end{align*}
		since $1 \leq \varrho(t,\xi)$.
	\end{proof}

	To describe the behavior of symbols in the pseudodifferential zone, we introduce the symbol class $\T_{N,\,M}(m_1,\,m_2,\,m_3)$.
	\begin{Definition}[$\T_{N,\,M}(m_1,\,m_2,\,m_3)$]
		Let $N,\,M > 0$ and $m_1,\,m_2,\,m_3 \in \R$. A function $a \in C\big([0,\,T];\,C^\infty(\R^n \times \R^n)\big)$ belongs to $\T_{N,\,M}(m_1,\,m_2,\,m_3)$ if
		\begin{equation*}
		\big|D_x^\beta \partial_\xi^\alpha a(t,\,x,\,\xi)\big| \leq C_{\alpha,\,\beta} \varrho(t,\,\xi)^{m_1} \Big(\frac{\partial_t \varrho(t,\,\xi)}{\varrho(t,\,\xi)}\Big)^{m_2} \jxi^{m_3-|\alpha|},
		\end{equation*}
		for all $\alpha,\,\beta \in \N^n$ and all $(t,\,x,\,\xi) \in \Zpd(N,\,M)$.
	\end{Definition}

	To describe the behavior of symbols in the hyperbolic zone, we introduce the symbol class $\S_{N,\,M}(l_1,\,l_2,\,l_3,\,l_4,\,l_5)$.
	\begin{Definition}[$\S_{N,\,M}(l_1,\,l_2,\,l_3,\,l_4,\,l_5)$]
		Let $N,\,M > 0$ and $l_1,\,\ldots,\,l_5 \in \R$. A function $a \in C\big([0,\,T];\,C^\infty(\R^n \times \R^n)\big)$ belongs to $\S_{N,\,M}(l_1,\,l_2,\,l_3,\,l_4,\,l_5)$ if
		\begin{equation*}
		\big|D_x^\beta \partial_\xi^\alpha a(t,\,x,\,\xi)\big| \leq C_{\alpha,\,\beta}\jxi^{l_1-|\alpha|} \lambda(t)^{l_2} \Big(\frac{\lambda(t)}{\Lambda(t)}\Big)^{l_3} (w(\Lambda(t)))^{l_4} (\vp(\jxi))^{l_5},
		\end{equation*}
		for all $\alpha,\,\beta \in \N^n$ and all $(t,\,x,\,\xi) \in \Zhyp(N,\,M)$.
	\end{Definition}

	\begin{Remark}
		Let us explain why we define the above mentioned symbol classes $\T_{N,\,M}(m_1,\,m_2,\,m_3)$ and $\S_{N,\,M}(l_1,\,l_2,\,l_3,\,l_4,\,l_5)$ in the way we do.
		
		In the hyperbolic zone $\Zhyp(N,\,M)$ where our problem can be treated similarly to a strictly hyperbolic problem, the behavior of the characteristic roots $\tau_k = \tau_k(t,\,x,\,\xi)$, $k = 1,\,\ldots,\,m$, of the original equation is important. In $\Zhyp(N,\,M)$ they basically behave like $\jxi \lambda(t)$, thus it is useful to include these terms in the definition of $\S_{N,\,M}(l_1,\,l_2,\,l_3,\,l_4,\,l_5)$. Furthermore, we also use the Levi condition \ref{CP2:LeviCond} in $\Zhyp(N,\,M)$ to deal with the lower-order terms which gives rise to terms behaving like $\lambda(t) (w(\Lambda(t)))^m$. The term $\vp = \vp(\jxi)$ appears due to the regularization we perform in the hyperbolic zone to account for the low regularity of the coefficients. Lastly, we note that terms that behave like $\frac{\lambda(t)}{\Lambda(t)}$ appear since we use $\jxi \lambda(t)$ in the definition of the energy for the hyperbolic zone. Applying a time derivative to these terms in the energy yields $\frac{\lambda(t)}{\Lambda(t)}$ (if we use the relations given in Definition~\ref{SF:DEF}). We note that all of these five terms appear for different reasons and that it is useful to keep track of them separately.
		
		To explain the symbol class $\T_{N,\,M}(m_1,\,m_2,\,m_3)$, we first note that we typically work in classes $\T_{2N,\,M}(m_1,\,m_2,\,m_3)$ which are related to the pseudodifferential zone $\Zpd(2N,\,M)$. This is helpful since in that way there is some overlap between the hyperbolic zone $\Zhyp(N,\,M)$ and the pseudodifferential zone $\Zpd(2N,\,M)$. Our approach for the treatment in the pseudodifferential zone is to include terms that behave like $\varrho = \varrho(t,\xi)$ in the energy for the pseudodifferential zone. By doing so, we obtain terms that behave like $\frac{\partial_t \varrho(t,\xi)}{\varrho(t,\xi)}$ from deriving the energy with respect to time. It is helpful to keep the terms $\varrho$ and $\frac{\partial_t \varrho}{\varrho}$ separate in the definition of the symbol class. One example for this are the computations in the area where both zones overlap. In the set $\Zhyp(N,\,M) \cap \Zpd(2N,\,M)$ we can show that $\varrho(t,\xi) \sim \jxi \lambda(t)$ and  $\frac{\partial_t \varrho(t,\xi)}{\varrho(t,\xi)} \sim\frac{\lambda(t)}{\Lambda(t)}$ which is helpful when calculating products of symbols.
	\end{Remark}

	\begin{Remark}
		We note that we deliberately do not include conditions on $D_t a(t,\,x,\,\xi)$ in the definition of $\S_{N,\,M}(l_1,\,l_2,\,l_3,\,l_4,\,l_5)$. Usually (with coefficients regular in time) one would expect a condition like
		\begin{equation*}
		\begin{aligned}[t]
		\big|D_t^j D_x^\beta \partial_\xi^\alpha a(t,\,x,\,\xi)\big| \leq C_{\alpha,\,\beta,\,j}\jxi^{l_1-|\alpha|} \lambda(t)^{l_2}\Big(\frac{\lambda(t)}{\Lambda(t)}\Big)^{l_3+j} (w(\Lambda(t)))^{m l_4}\\\times (\vp(\jxi))^{l_5},
		\end{aligned}
		\end{equation*}
		for all $\alpha,\,\beta \in \N^n$, $j \leq j_\text{max}$ and all $(t,\,x,\,\xi) \in \Zhyp(N,\,M)$, for some $j_\text{max} \in \N$. However, since the coefficients are low-regular in time, we come across some symbols that would not fit into this classification. For that reason, we omit a condition on $D_t a(t,\,x,\,\xi)$ in the definition and characterize the behavior of the time-derivatives directly, e.g. by
		\begin{equation*}
		D_t^j a \in \S_{N,\,M}(l_1,\,l_2,\,l_3+j,\,l_4,\,l_5) + \S_{N,\,M}(l_1,\,l_2,\,l_3,\,l_4,\,l_5+j),
		\end{equation*}
		when needed.
	\end{Remark}
	\begin{Remark}
		We often employ cut-off functions to restrict symbols to a certain zone. Usually, we use a function $\chi \in C^\infty$ with
		\begin{equation*}
			\chi(\xi) = \begin{cases}
				1,& |\xi| \leq 1,\\
				0,& |\xi| \geq 2,
				\end{cases}
		\end{equation*}
		and $0 \leq \chi(\xi) \leq 1$. To restrict a symbol to the pseudodifferential zone $\Zpd(N,\,M)$ we use $\chi\big(\frac{\jxi}{\frac{N}{2} (w(\Lambda(t)))^m}\big)$. Note, that we have $\frac{N}{2}$ in the denominator to ensure that $\chi \equiv 0$, when we leave the zone. Using $\chi\big(\frac{\jxi}{N (w(\Lambda(t)))^m}\big)$ restricts a symbol to $\Zpd(2N,\,M)$.
		To restrict a symbol to the hyperbolic zone $\Zhyp(N,\,M)$ we use $\big(1-\chi\big(\frac{\jxi}{N (w(\Lambda(t)))^m}\big)\big)$.
	\end{Remark}

	\begin{Remark}
		We note that symbols that are identically zero in $\Zpd(N,\,M)$ belong to $\T_{N,\,M}(0,\,0,\,-\infty)$.
	\end{Remark}

	The following properties and symbol hierarchies are due to the definitions of the zones and the definitions of the symbol classes. They are obtained by straightforward computations.
	\begin{Proposition}[Symbol hierarchies and properties]\label{Prop:SymbolProp}
		Let $N,\,M > 0$, $\alpha,\,\beta \in \N^n$ and $m_1,\,m_2,\,m_3,\,l_1,\,\ldots,\,l_5 \in \R$. We then have
		\begin{enumerate}[label = (\roman*),align = left, leftmargin=*]
			\item If $a \in \T_{N,\,M}(m_1,\,m_2,\,m_3)$, then $D_x^\beta a \in \T_{N,\,M}(m_1,\,m_2,\,m_3)$.
			\item If $a \in \T_{N,\,M}(m_1,\,m_2,\,m_3)$, then $\partial_\xi^\alpha a \in \T_{N,\,M}(m_1,\,m_2,\,m_3-|\alpha|)$.
			\item If $a \in \T_{N,\,M}(m_1,\,m_2,\,m_3)$ and $\widetilde a \in \T_{N,\,M}(\widetilde m_1,\,\widetilde m_2,\,\widetilde m_3)$, then \\
$ \widetilde a a \in \T_{N,\,M}(m_1+\widetilde m_1,\,m_2+ \widetilde m_2,\,m_3+ \widetilde m_3)$.
			\item If $a \equiv 0$ in $\Zpd(N,\,M)$, then $a \in \T_{N,\,M}(0,\,0,\,-\infty)$.
			\item We have $\T_{N,\,M}(m_1,\,m_2,\,m_3 -|\alpha|) \subset \T_{N,\,M}(m_1,\,m_2,\,m_3)$.
			\item If $N_1\leq N,\, M_1 \geq M$, then\\ $\T_{N,\,M}(m_1,\,m_2,\,m_3) \subset \T_{N_1,\,M_1}(m_1,\,m_2,\,m_3)$.			
			\item If $a\in \S_{N,\,M}(l_1,\,l_2,\,l_3,\,l_4,\,l_5)$, then $D_x^\beta a\in \S_{N,\,M}(l_1,\,l_2,\,l_3,\,l_4,\,l_5)$.
			\item If $a\in \S_{N,\,M}(l_1,\,l_2,\,l_3,\,l_4,\,l_5)$, then $\partial_\xi^\alpha a\in \S_{N,\,M}(l_1-|\alpha|,\,l_2,\,l_3,\,l_4,\,l_5)$.
			\item If $a\in \S_{N,\,M}(l_1,\,l_2,\,l_3,\,l_4,\,l_5)$ and $\widetilde a\in \S_{N,\,M}(\widetilde l_1,\,\widetilde l_2,\,\widetilde l_3,\,\widetilde l_4,\,\widetilde l_5)$, then \\ $a\widetilde a\in \S_{N,\,M}(l_1+\widetilde l_1,\,l_2+\widetilde l_2,\,l_3+\widetilde l_3,\,l_4+\widetilde l_4,\,l_5+\widetilde l_5)$.
			\item We have $\S_{N,\,M}(l_1,\,l_2,\,l_3+k,\,l_4,\,l_5) \subset \S_{N,\,M}(l_1,\,l_2+k,\,l_3,\,l_4+k,\,l_5)$, for $k \geq 0$.
			\item We have $\S_{N,\,M}(l_1,\,l_2,\,l_3,\,l_4,\,l_5) \subset \S_{N,\,M}(l_1+k_1+k_2,\,l_2,\,l_3,\,l_4-k_1,\,l_5-k_2)$, for $k_1,\,k_2 \geq 0$.
			\item If $N\leq N_1,\,M_1 \geq M$, then\\ $\S_{N,\,M}(l_1,\,l_2,\,l_3,\,l_4,\,l_5) \subset \S_{N_1,\,M_1}(l_1,\,l_2,\,l_3,\,l_4,\,l_5)$.
		\end{enumerate}
	\end{Proposition}
	
	The following two propositions give relations to the standard symbol classes $\Sy^m$. We use these relations later on to explain composition, parametrix and adjoint operators.
	
	\begin{Proposition}\label{Prop:StoS}
		Let $ a = a(t,\,x,\,\xi)$ be a symbol with
		\begin{equation*}
		a \in \T_{N,\,M}(0,\,0,\,0) + \S_{N,\,M}(l_1,\,l_2,\,l_3,\,l_4,\,l_5).
		\end{equation*} Then we have
		\begin{equation*}
		 a \in L^\infty\big([0,\,T];\,\Sy^{\max\{0,\,l_1+l_3+l_4+l_5\}}\big).
		\end{equation*}
	\end{Proposition}
	\begin{proof}
		The definition of the zones, Proposition~\ref{CP2:Remark:EstLambdaW} and straightforward calculations yield
		\begin{align*}
		 \big|D_x^\beta \partial_\xi^\alpha a(t,\,x,\,\xi)\big| &\begin{aligned}[t] \leq C_{\alpha,\,\beta} \jxi^{-|\alpha|} +  C_{\alpha,\,\beta}\jxi^{l_1-|\alpha|} \lambda(t)^{l_2} \Big(\frac{\lambda(t)}{\Lambda(t)}\Big)^{l_3} (w(\Lambda(t)))^{m l_4}\\\times (\vp(\jxi))^{l_5}
		 \end{aligned}\\
		 & \leq C_{\alpha,\,\beta} \jxi^{-|\alpha|} +  C_{\alpha,\,\beta}\jxi^{l_1-|\alpha| + l_4 + l_5} \lambda(t)^{l_2+l_3} \Big(\frac{1}{\Lambda(t)}\Big)^{l_3}\\
		 & \leq C_{\alpha,\,\beta} \jxi^{-|\alpha|} +  C_{\alpha,\,\beta}\jxi^{l_1-|\alpha| + l_4 + l_5} \lambda(t)^{l_2+l_3} (w(\Lambda(t)))^{ml_3}\\
		 & \leq C_{\alpha,\,\beta} \jxi^{-|\alpha|} +  C_{\alpha,\,\beta}\jxi^{l_1-|\alpha| + l_3+l_4 + l_5} \lambda(t)^{l_2+l_3}.
		\end{align*}
	\end{proof}

	\begin{Proposition}\label{Prop:TtoS}
		Let $ a = a(t,\,x,\,\xi)$ be a symbol with
		\begin{equation*}
		a \in \T_{N,\,M}(m_1,\,m_2,\,m_3) + \S_{N,\,M}(0,\,0,\,0,\,0,\,0).
		\end{equation*}
		Then we have
		\begin{equation*}
		a \in L^\infty_\text{loc}\big((0,\,T];\,\Sy^{\max\{0,\,\frac{m_1}{m} + m_2+m_3,\,m_2+m_3\}}\big).
		\end{equation*}
	\end{Proposition}
	\begin{proof}
		The definition of the zones, Lemma~\ref{Lemma:rho} and straightforward calculations yield
		\begin{align*}
			|D_x^\beta \partial_\xi^\alpha a(t,\,x,\,\xi)| & \leq C_{\alpha,\,\beta} \varrho^{m_1}(t,\,\xi) \Big(\frac{\partial_t \varrho(t,\,\xi)}{\varrho(t,\,\xi)}\Big)^{m_2}\jxi^{m_3-|\alpha|} +  C_{\alpha,\,\beta}\jxi^{-|\alpha|} \\
			&\begin{aligned}[t]
				\leq C_{\alpha,\,\beta}& \Big(1+ \jxi \lambda^m(t) (w(\Lambda(t)))^{m(m-1)}\Big)^{\frac{m_1}{m}}\\&\times \Big(\lambda^m(t) \frac{\lambda(t)}{\Lambda(t)} (w(\Lambda(t)))^{m(m-1)} \jxi\Big)^{m_2} \jxi^{m_3-|\alpha|}\\ &+  C_{\alpha,\,\beta}\jxi^{-|\alpha|}
				\end{aligned}\\
			&\begin{aligned}[t]
				\leq C_{\alpha,\,\beta}& \Big(1+ \jxi^{\frac{m_1}{m}} \lambda^{m_1}(t) (w(\Lambda(t)))^{m_1(m-1)}\Big)\\&\times \Big(\lambda^m(t) \frac{\lambda(t)}{\Lambda(t)} (w(\Lambda(t)))^{m(m-1)}\Big)^{m_2} \jxi^{m_2+m_3-|\alpha|}\\ &+  C_{\alpha,\,\beta}\jxi^{-|\alpha|}.
			\end{aligned}
		\end{align*}
	From this we conclude the assertion of the lemma, since the terms
	\begin{equation*}
	\lambda^{m_1}(t) \big(w(\Lambda(t))\big)^{m_1(m-1)}\text{ and }\lambda^m(t) \frac{\lambda(t)}{\Lambda(t)} \big(w(\Lambda(t))\big)^{m(m-1)},
	\end{equation*}
	are bounded on any interval $[\ve,\,T],\,\ve > 0$.
	\end{proof}

	\begin{Lemma}[Asymptotic expansion]\label{Lemma:AsymExp}
		Let $\{a_j\}_j$ be a sequence of symbols with
		\begin{equation*}
		a_j \in \T_{N,\,M}(m_1,\,m_2,\,m_3-j) + \S_{N,\,M}(l_1-j,\,l_2,\,l_3,\,l_4,\,l_5),\,j \geq 0.
		\end{equation*}
		Then there is a symbol
		\begin{equation*}
		a \in T_{N,\,M}(m_1,\,m_2,\,m_3) + \S_{N,\,M}(l_1,\,l_2,\,l_3,\,l_4,\,l_5),
		\end{equation*}
		such that
		\begin{equation*}
			a(t,\,x,\,\xi) \sim \sum\limits_{j = 0}^\infty a_j(t,\,x,\,\xi),
		\end{equation*}
	that is
	\begin{equation*}
	\begin{aligned}
		a(t,\,x,\,\xi) - \sum\limits_{j = 0}^{j_0-1} a_j(t,\,x,\,\xi) &\in \T_{N,\,M}(m_1,\,m_2,\,m_3-j_0)\\ &+ \S_{N,\,M}(l_1-j_0,\,l_2,\,l_3,\,l_4,\,l_5),
		\end{aligned}
	\end{equation*}
	for all $j_0 \geq 1$. The symbol is uniquely determined modulo $L^\infty_\text{loc}\big((0,\,T];\,\Sy^{-\infty}\big)$.
	\end{Lemma}
	\begin{proof}
		Let $\chi$ be a $C^\infty$ cut-off function with
		\begin{equation*}
			\chi(\xi) = \begin{cases}
				1, & |\xi| \leq 1,\\
				0,& |\xi| \geq 2,
				\end{cases}
		\end{equation*}
		and $0 \leq \chi \leq 1$. For a sequence of positive numbers $\ve_j\rightarrow 0$, we define
		\begin{equation*}
			\gamma_{\ve_j}(\xi) = 1- \chi(\ve_j \xi).
		\end{equation*}
		We note that $\gamma_{\ve_j}(\xi) = 0$ if $|\xi| < \frac{1}{\ve_j}$. We choose $\ve_j$ such that
		\begin{equation*}
			\ve_j \leq 2^{-j},
		\end{equation*}
		and set
		\begin{equation*}
			a(t,\,x,\,\xi) = \sum\limits_{j=0}^{\infty} \gamma_{\ve_j}(\xi) a_j(t,\,x,\,\xi).
		\end{equation*}
		We note that $a(t,\,x,\,\xi)$ exists (i.e. the series converges pointwise), since for any fixed point $(t,\,x,\,\xi)$ only a finite number of summands contribute to $a(t,\,x,\,\xi)$. Indeed, for fixed $(t,\,x,\,\xi)$ we can always find a $j_0$ such that $|\xi| < \frac{1}{\ve_{j_0}}$ and so
		\begin{equation*}
			a(t,\,x,\,\xi) = \sum \limits_{j = 0}^{j_0-1} \gamma_{\ve_j}(\xi) a_j(t,\,x,\,\xi).
		\end{equation*}
	
		We continue by showing that the symbol \begin{equation*}
		a\in T_{N,\,M}(m_1,\,m_2,\,m_3) + \S_{N,\,M}(l_1,\,l_2,\,l_3,\,l_4,\,l_5).
		\end{equation*}
		To prove this, we first observe that
		\begin{align*}
	\Big|D_x^\beta \partial_\xi^\alpha \big(\gamma_{\ve_j}(\xi) a_j(t,\,x,\,\xi)\big)\Big| & = \bigg| \sum\limits_{\alpha^\prime + \alpha^{\prime\prime} = \alpha} \binom{\alpha}{\alpha^\prime} \partial_\xi^{\alpha^\prime} \gamma_{\ve_j}(\xi) D_x^\beta \partial_\xi^{\alpha^{\prime\prime}} a_j(t,\,x,\,\xi)\bigg|\\
		& \leq \bigg| \gamma_{\ve_j}(\xi) D_x^\beta \partial_\xi^{\alpha} a_j(t,\,x,\,\xi)\\& + \sum\limits_{\substack{\alpha^\prime + \alpha^{\prime\prime} = \alpha\\|\alpha^\prime| > 0}} C_{\alpha^\prime} \widetilde \chi_{\ve_j}(\xi)\jxi^{-|\alpha^\prime|} D_x^\beta  \partial_\xi^{\alpha^{\prime\prime}} a_j(t,\,x,\,\xi)\bigg|,
		\end{align*}
		where $\widetilde \chi_{\ve_j}(\xi)$ is another smooth cut-off function which is non-zero only if $1 \leq \ve_j |\xi| \leq 2$. This new cut-off function describes the support of the derivatives of $\gamma_{\ve_j}(\xi)$. In the last estimate, we also used that $\ve_j \sim \jxi$ if $\widetilde \chi_{\ve_j}(\xi) \neq 0$.
		We conclude that
		\begin{equation*}
		 \begin{aligned}[t]
		\Big|D_x^\beta \partial_\xi^\alpha \big(\gamma_{\ve_j}(\xi)& a_j(t,\,x,\,\xi)\big)\Big|\\
		 &\leq
		\bigg|2^{-j} \bigg[\varrho^{m_1} \Big(\frac{\partial_t \varrho}{\varrho}\Big)^{m_2} \jxi^{m_3+1-j-|\alpha|} \chi\Big(\frac{\jxi}{\frac{N}{2}(w(\Lambda(t)))^m}\Big)\\& + \jxi^{l_1+1-j-|\alpha|} \lambda^{l_2}(t) \Big(\frac{\lambda(t)}{\Lambda(t)}\Big)^{l_3} (w(\Lambda(t)))^{ml_4} \\ &\times(\vp(\jxi))^{l_5} \Big(1 - \chi\Big(\frac{\jxi}{N(w(\Lambda(t)))^m}\Big)\Big)\bigg]\bigg|,
		\end{aligned}
		\end{equation*}
		where we used that $a_j $ belongs to
		\begin{equation*}
		\T_{N,\,M}(m_1,\,m_2,\,m_3-j) + \S_{N,\,M}(l_1-j,\,l_2,\,l_3,\,l_4,\,l_5),
		\end{equation*}
		and estimated $|\xi| \geq 2^{j}$ (due to the support of cut-off functions) once in each summand.
		Using this relation, we obtain
		\begin{align*}
		\big|D_x^\beta \partial_\xi^\alpha &a(t,\,x,\,\xi)\big|\\  &\leq \sum\limits_{j=0}^{j_0-1}\Big|D_x^\beta \partial_\xi^\alpha \big(\gamma_{\ve_j}(\xi) a_j(t,\,x,\,\xi)\big)\Big|\\
		&\leq \begin{aligned}[t]
		 &C_{\alpha,\,\beta} \bigg[\varrho^{m_1} \Big(\frac{\partial_t \varrho}{\varrho}\Big)^{m_2} \jxi^{m_3-|\alpha|} \chi\Big(\frac{\jxi}{\frac{N}{2}(w(\Lambda(t)))^m}\Big)\\& + \jxi^{l_1-|\alpha|} \lambda^{l_2}(t) \Big(\frac{\lambda(t)}{\Lambda(t)}\Big)^{l_3} (w(\Lambda(t)))^{m l_4} \\ &\times(\vp(\jxi))^{l_5} \Big(1 - \chi\Big(\frac{\jxi}{N(w(\Lambda(t)))^m}\Big)\Big)\bigg]\\
		 &+
		 \sum\limits_{j=1}^{j_0-1} 2^{-j} \bigg[\varrho^{m_1} \Big(\frac{\partial_t \varrho}{\varrho}\Big)^{m_2} \jxi^{m_3+1-j-|\alpha|} \chi\Big(\frac{\jxi}{\frac{N}{2}(w(\Lambda(t)))^m}\Big)\\& + \jxi^{l_1+1-j-|\alpha|} \lambda^{l_2}(t) \Big(\frac{\lambda(t)}{\Lambda(t)}\Big)^{l_3} (w(\Lambda(t)))^{m l_4} \\ &\times(\vp(\jxi))^{l_5} \Big(1 - \chi\Big(\frac{\jxi}{N(w(\Lambda(t)))^m}\Big)\Big)\bigg]
		\end{aligned}\\
		& \leq \begin{aligned}[t]
		&C_{\alpha,\,\beta} \bigg[\varrho^{m_1} \Big(\frac{\partial_t \varrho}{\varrho}\Big)^{m_2} \jxi^{m_3-|\alpha|} \chi\Big(\frac{\jxi}{\frac{N}{2}(w(\Lambda(t)))^m}\Big)\\& + \jxi^{l_1-|\alpha|} \lambda^{l_2}(t) \Big(\frac{\lambda(t)}{\Lambda(t)}\Big)^{l_3} (w(\Lambda(t)))^{m l_4} \\ &\times(\vp(\jxi))^{l_5} \Big(1 - \chi\Big(\frac{\jxi}{N(w(\Lambda(t)))^m}\Big)\Big)\bigg].
		\end{aligned}
		\end{align*}
		
		As for the remainder of the series, we have
		\begin{align*}
		&\Big| D_x^\beta \partial_\xi^\alpha \Big(\sum\limits_{j = j_0}^{\infty} \gamma_{\ve_j}(\xi) a_j(t,\,x,\,\xi)\Big)\Big|\\
		&\leq \Big| D_x^\beta \partial_\xi^\alpha \big(\gamma_{\ve_{j_0}}(\xi) a_{j_0}(t,\,x,\,\xi)\big)\Big| + \sum\limits_{j = j_0+1}^{\infty}\Big|D_x^\beta \partial_\xi^\alpha \big(\gamma_{\ve_j}(\xi) a_j(t,\,x,\,\xi)\big)\Big|\\
		&\leq \begin{aligned}[t]
		&C_{\alpha,\,\beta} \bigg[\varrho^{m_1} \Big(\frac{\partial_t \varrho}{\varrho}\Big)^{m_2} \jxi^{m_3-j_0-|\alpha|} \chi\Big(\frac{\jxi}{\frac{N}{2}(w(\Lambda(t)))^m}\Big)\\& + \jxi^{l_1-j_0-|\alpha|} \lambda^{l_2}(t) \Big(\frac{\lambda(t)}{\Lambda(t)}\Big)^{l_3} (w(\Lambda(t)))^{m l_4} \\ &\times(\vp(\jxi))^{l_5} \Big(1 - \chi\Big(\frac{\jxi}{N(w(\Lambda(t)))^m}\Big)\Big)\bigg]\\
		&+
		\sum\limits_{j=j_0+1}^{\infty} 2^{-j} \bigg[\varrho^{m_1} \Big(\frac{\partial_t \varrho}{\varrho}\Big)^{m_2} \jxi^{m_3+1-j-|\alpha|} \chi\Big(\frac{\jxi}{\frac{N}{2}(w(\Lambda(t)))^m}\Big)\\& + \jxi^{l_1+1-j-|\alpha|} \lambda^{l_2}(t) \Big(\frac{\lambda(t)}{\Lambda(t)}\Big)^{l_3} (w(\Lambda(t)))^{m l_4} \\ &\times(\vp(\jxi))^{l_5} \Big(1 - \chi\Big(\frac{\jxi}{N(w(\Lambda(t)))^m}\Big)\Big)\bigg]
		\end{aligned}\\
		&\leq \begin{aligned}[t]
		&C_{\alpha,\,\beta} \bigg[\varrho^{m_1} \Big(\frac{\partial_t \varrho}{\varrho}\Big)^{m_2} \jxi^{m_3-j_0-|\alpha|} \chi\Big(\frac{\jxi}{\frac{N}{2}(w(\Lambda(t)))^m}\Big)\\& + \jxi^{l_1-j_0-|\alpha|} \lambda^{l_2}(t) \Big(\frac{\lambda(t)}{\Lambda(t)}\Big)^{l_3} (w(\Lambda(t)))^{m l_4} \\ &\times(\vp(\jxi))^{l_5} \Big(1 - \chi\Big(\frac{\jxi}{N(w(\Lambda(t)))^m}\Big)\Big)\bigg].
		\end{aligned}
		\end{align*}
		Thus,
		\begin{equation*}
		\begin{aligned}
		a(t,\,x,\,\xi) - \sum\limits_{j = 0}^{j_0-1} a_j(t,\,x,\,\xi) &\in \T_{N,\,M}(m_1,\,m_2,\,m_3-j_0)\\ &+ \S_{N,\,M}(l_1-j_0,\,l_2,\,l_3,\,l_4,\,l_5).
		\end{aligned}
		\end{equation*}
		
		Lastly, we use Proposition~\ref{Prop:StoS} and Proposition~\ref{Prop:TtoS} to conclude that if a symbol
		\begin{equation*}
		a_j \in \T_{N,\,M}(m_1,\,m_2,\,m_3-j) + \S_{N,\,M}(l_1-j,\,l_2,\,l_3,\,l_4,\,l_5),
		\end{equation*}
		then also
		\begin{equation*}
		 a_j \in L^\infty_\text{loc}\big((0,\,T];\,\Sy^{\max\{l_1-j+l_3+l_4+l_5,\,m_2+m_3-j,\,\frac{m_1}{m} + m_2 + m_3-j\}}\big).
		\end{equation*} If $j$ tends to $+\infty$, then the intersection of all those spaces belongs to the space $L^\infty_\text{loc}\big((0,\,T];\,\Sy^{-\infty}\big)$. This completes the proof.
	\end{proof}

	\begin{Lemma}\label{Lemma:SymbolProd}
		Let
		\begin{equation*}
			a \in \T_{2N,\,M}(m_1,\,m_2,\,m_3) + \S_{N,\,M}(l_1,\,l_2,\,l_3,\,l_4,\,l_5),
		\end{equation*}
		and
		\begin{equation*}
		\widetilde a \in \T_{2N,\,M}(\widetilde m_1,\,\widetilde m_2,\,\widetilde m_3) + \S_{N,\,M}(\widetilde l_1,\,\widetilde l_2,\,\widetilde l_3,\,\widetilde l_4,\,\widetilde l_5),
		\end{equation*}
		then
		\begin{equation*}
			\begin{aligned}[t]
			a \widetilde a &\in  \T_{2N,\,M}(m_1+\widetilde m_1,\,m_2+\widetilde m_2,\,m_3+\widetilde m_3)\\ &+ \S_{N,\,M}(l_1+\widetilde l_1 ,\,l_2+\widetilde l_2,\,l_3+ \widetilde l_3,\,l_4+\widetilde l_4,\,l_5+\widetilde l_5) \\ &+  \S_{N,\,M}(m_1+m_3+\widetilde l_1 ,\,m_1+\widetilde l_2,\,m_2+ \widetilde l_3,\,\widetilde l_4,\,\widetilde l_5) \\ &+  \S_{N,\,M}(\widetilde m_1 + \widetilde m_3+ l_1 ,\,\widetilde m_1 + l_2,\,\widetilde m_2+l_3,\,l_4,\,l_5).
			\end{aligned}
		\end{equation*}
	\end{Lemma}
	\begin{proof}
		We obtain by straightforward computation that
		\begin{align*}
			\big|&D_x^\beta \partial_\xi^\alpha a(t,\,x,\,\xi) \widetilde a(t,\,x,\,\xi)\big|\\ & \leq \begin{aligned}[t]
				\sum\limits_{\substack{\beta^\prime + \beta^{\prime\prime} = \beta\\\alpha^\prime+\alpha^{\prime\prime} = \alpha}} C_{\alpha^\prime,\,\beta^\prime} |D_x^{\beta^\prime} \partial_\xi^{\alpha^\prime} a(t,\,x,\,\xi)| |D_x^{\beta^{\prime\prime}} \partial_\xi^{\alpha^{\prime\prime}}  \widetilde a(t,\,x,\,\xi)|
				\end{aligned}\\
			&\leq \begin{aligned}[t]
				& C_{\alpha,\,\beta} \Big[\varrho^{m_1+\widetilde m_1} \Big(\frac{\partial_t \varrho}{\varrho}\Big)^{m_2+\widetilde m_2} \jxi^{m_3 + \widetilde m_3 - |\alpha|}\chi\Big(\frac{\jxi}{N(w(\Lambda(t)))^m}\Big)\\
				&+ \jxi^{l_1+\widetilde l_1 - |\alpha|} \lambda(t)^{l_2+\widetilde l_2} \Big(\frac{\lambda(t)}{\Lambda(t)}\Big)^{l_3 + \widetilde l_3} (w(\Lambda(t)))^{m(l_4 + \widetilde l_4)} (\vp(\jxi))^{l_5+\widetilde l_5}\\ &\times\Big(1- \chi\Big(\frac{\jxi}{N(w(\Lambda(t)))^m}\Big)\Big)\\
				&+ \varrho^{m_1} \Big(\frac{\partial_t \varrho}{\varrho}\Big)^{m_2} \jxi^{m_3 + \widetilde l_1 - |\alpha|} \lambda(t)^{\widetilde l_2} \Big(\frac{\lambda(t)}{\Lambda(t)}\Big)^{\widetilde l_3} (w(\Lambda(t)))^{m\widetilde l_4} (\vp(\jxi))^{\widetilde l_5} \\&\times\chi\Big(\frac{\jxi}{N(w(\Lambda(t)))^m}\Big)\Big(1- \chi\Big(\frac{\jxi}{N(w(\Lambda(t)))^m}\Big)\Big)\\
				&+ \varrho^{\widetilde m_1} \Big(\frac{\partial_t \varrho}{\varrho}\Big)^{\widetilde m_2} \jxi^{\widetilde m_3 + l_1- |\alpha|} \lambda(t)^{l_2} \Big(\frac{\lambda(t)}{\Lambda(t)}\Big)^{l_3 } (w(\Lambda(t)))^{ml_4 } (\vp(\jxi))^{l_5} \\&\times\chi\Big(\frac{\jxi}{N(w(\Lambda(t)))^m}\Big)\Big(1- \chi\Big(\frac{\jxi}{N(w(\Lambda(t)))^m}\Big)\Big)
				\Big].
			\end{aligned}
		\end{align*}
		We note that the last two summands of the above inequality are only non-zero, if $N(w(\Lambda(t))^m < \jxi < 2N(w(\Lambda(t)))^m$ and that, therefore,
		\begin{align*}
			\varrho(t,\,\xi) &= (1+\jxi\lambda(t)^m (w(\Lambda(t)))^m)^{\frac{1}{m}}\leq C \lambda(t) \jxi,\\
			\frac{\partial_t \varrho(t,\,\xi)}{\varrho(t,\,\xi)} & \leq C \frac{\jxi \lambda(t)^m \frac{\lambda(t)}{\Lambda(t)} (w(\Lambda(t)))^{m(m-1)}}{1+\jxi\lambda(t)^m (w(\Lambda(t)))^{m(m-1)}}\leq C \frac{\lambda(t)}{\Lambda(t)},
		\end{align*}
	for $N(w(\Lambda(t))^m < \jxi < 2N(w(\Lambda(t)))^m$.
	This yields
	\begin{align*}
	\big|&D_x^\beta \partial_\xi^\alpha a(t,\,x,\,\xi) \widetilde a(t,\,x,\,\xi)\big|\\
	&\leq \begin{aligned}[t]
		& C_{\alpha,\,\beta} \Big[\varrho^{m_1+\widetilde m_1} \Big(\frac{\partial_t \varrho}{\varrho}\Big)^{m_2+\widetilde m_2} \jxi^{m_3 + \widetilde m_3 - |\alpha|}\chi\Big(\frac{\jxi}{N(w(\Lambda(t)))^m}\Big)\\
		&+ \jxi^{l_1+\widetilde l_1 - |\alpha|} \lambda(t)^{l_2+\widetilde l_2} \Big(\frac{\lambda(t)}{\Lambda(t)}\Big)^{l_3 + \widetilde l_3} (w(\Lambda(t)))^{m(l_4 + \widetilde l_4)} (\vp(\jxi))^{l_5+\widetilde l_5}\\ &\times\Big(1- \chi\Big(\frac{\jxi}{N(w(\Lambda(t)))^m}\Big)\Big)\\
		&+ \jxi^{m_1+m_3 + \widetilde l_1 - |\alpha|} \lambda(t)^{m_1+\widetilde l_2} \Big(\frac{\lambda(t)}{\Lambda(t)}\Big)^{m_2+\widetilde l_3} (w(\Lambda(t)))^{m\widetilde l_4} (\vp(\jxi))^{\widetilde l_5} \\&\times\chi\Big(\frac{\jxi}{N(w(\Lambda(t)))^m}\Big)\Big(1- \chi\Big(\frac{\jxi}{N(w(\Lambda(t)))^m}\Big)\Big)\\
		&+  \jxi^{\widetilde m_1+\widetilde m_3 + l_1- |\alpha|} \lambda(t)^{\widetilde m_1+l_2} \Big(\frac{\lambda(t)}{\Lambda(t)}\Big)^{\widetilde m_2+l_3 } (w(\Lambda(t)))^{ml_4 } (\vp(\jxi))^{l_5} \\&\times\chi\Big(\frac{\jxi}{N(w(\Lambda(t)))^m}\Big)\Big(1- \chi\Big(\frac{\jxi}{N(w(\Lambda(t)))^m}\Big)\Big)
		\Big],
	\end{aligned}
	\end{align*}
	so that we conclude
	\begin{align*}
		a \widetilde a &\in  \T_{2N,\,M}(m_1+\widetilde m_1,\,m_2+\widetilde m_2,\,m_3+\widetilde m_3)\\ &+ \S_{N,\,M}(l_1+\widetilde l_1 ,\,l_2+\widetilde l_2,\,l_3+ \widetilde l_3,\,l_4+\widetilde l_4,\,l_5+\widetilde l_5) \\ &+  \S_{N,\,M}(m_1+m_3+\widetilde l_1 ,\,m_1+\widetilde l_2,\,m_2+ \widetilde l_3,\,\widetilde l_4,\,\widetilde l_5) \\ &+  \S_{N,\,M}(\widetilde m_1 + \widetilde m_3+ l_1 ,\,\widetilde m_1 + l_2,\,\widetilde m_2+l_3,\,l_4,\,l_5).
	\end{align*}
	\end{proof}
	\begin{Remark}
		If the symbols $a$ and $\widetilde a$ in the previous lemma belong to \begin{equation*}
		\T_{N,\,M}(m_1,\,m_2,\,m_3) + \S_{N,\,M}(l_1,\,l_2,\,l_3,\,l_4,\,l_5),
		\end{equation*}
		and
		\begin{equation*}
		\T_{N,\,M}(\widetilde m_1,\,\widetilde m_2,\,\widetilde m_3) + \S_{N,\,M}(\widetilde l_1,\,\widetilde l_2,\,\widetilde l_3,\,\widetilde l_4,\,\widetilde l_5),
		\end{equation*}
		respectively, then their product belongs to
		 \begin{equation*}
		 \T_{N,\,M}(m_1+\widetilde m_1,\,m_2 + \widetilde m_2,\,m_3 + \widetilde m_3) + \S_{N,\,M}(l_1 \widetilde l_1,\,l_2 + \widetilde l_2,\,l_3 + \widetilde l_3,\,l_4 + \widetilde l_4,\,l_5 + \widetilde l_5),
		 \end{equation*} since there is no overlap of the zones $\Zpd(N,\,M)$ and $\Zhyp(N,\,M)$.
	\end{Remark}

	\begin{Lemma}\label{Lemma:CompPseu}
		Let $A$ and $B$ be pseudodifferential operators with symbols 
		\begin{equation*}
		a = \sigma(A) \in \T_{2N,\,M}(m_1,\,m_2,\,m_3) + \S_{N,\,M}(l_1,\,l_2,\,l_3,\,l_4,\,l_5),
		\end{equation*}
		and
		\begin{equation*}
		b = \sigma(B) \in \T_{2N,\,M}(\widetilde m_1,\,\widetilde m_2,\,\widetilde m_3) + \S_{N,\,M}(\widetilde l_1,\,\widetilde l_2,\,\widetilde l_3,\,\widetilde l_4,\,\widetilde l_5).
		\end{equation*}
		Then the pseudodifferential operator $C = A \circ B$ has a symbol
		\begin{align*}
			c = \sigma(C) &\in \T_{2N,\,M}(m_1+\widetilde m_1,\,m_2+\widetilde m_2,\,m_3+\widetilde m_3)\\ &+ \S_{N,\,M}(l_1+\widetilde l_1 ,\,l_2+\widetilde l_2,\,l_3+ \widetilde l_3,\,l_4+\widetilde l_4,\,l_5+\widetilde l_5) \\ &+  \S_{N,\,M}(m_1+m_3+\widetilde l_1 ,\,m_1+\widetilde l_2,\,m_2+ \widetilde l_3,\,\widetilde l_4,\,\widetilde l_5) \\ &+  \S_{N,\,M}(\widetilde m_1 + \widetilde m_3+ l_1 ,\,\widetilde m_1 + l_2,\,\widetilde m_2+l_3,\,l_4,\,l_5),
		\end{align*}
		and satisfies
		\begin{equation}\label{Lemma:CompPseu:eq}
			c(t,\,x,\,\xi) \sim \sum\limits_{\alpha \in \N^n} \frac{1}{\alpha!} \partial_\xi^\alpha a(t,\,x,\,\xi) D_x^\alpha b(t,\,x,\,\xi).
		\end{equation}
	The operator $C$ is uniquely determined modulo an operator with symbol from $L^\infty_\text{loc}\big((0,\,T];\,\Sy^{-\infty}\big)$.
	\end{Lemma}
	\begin{proof}
		In view of Proposition~\ref{Prop:StoS} and Proposition~\ref{Prop:TtoS} it is clear that the operator $C$ is a well-defined pseudodifferential operator. Relation \eqref{Lemma:CompPseu:eq} is a direct consequence of the composition rules in $\OPS^m$. Applying Proposition~\ref{Prop:SymbolProp}, Lemma~\ref{Lemma:SymbolProd} and Lemma~\ref{Lemma:AsymExp} then yields the desired statements.
	\end{proof}

	\begin{Lemma}\label{Lemma:Parametrix}
		Let $A$ be a pseudodifferential operator with an invertible symbol 
		\begin{equation*}
		a = \sigma(A) \in \T_{2N,\,M}(0,\,0,\,0) + \S_{N,\,M}(0,\,0,\,0,\,0,\,0).
		\end{equation*}
		Then there exists a parametrix $A^\#$ with symbol
		\begin{equation*}
		a^\# = \sigma(A^\#) \in \T_{2N,\,M}(0,\,0,\,0) + \S_{N,\,M}(0,\,0,\,0,\,0,\,0).
		\end{equation*}
	\end{Lemma}
	\begin{proof}
		We use the existence of the inverse of $a$ and set
		\begin{equation*}
			a_0^\#(t,\,x,\,\xi) = a(t,\,x,\,\xi)^{-1} \in \T_{2N,\,M}(0,\,0,\,0) + \S_{N,\,M}(0,\,0,\,0,\,0,\,0).
		\end{equation*}
		In view of Proposition~\ref{Prop:StoS} and Proposition~\ref{Prop:TtoS}, we are able to define a sequence $a_j^\#(t,\,x,\,\xi)$ recursively by
		\begin{equation*}
			\sum\limits_{1 \leq |\alpha| \leq j} \frac{1}{\alpha!} \partial_\xi^\alpha a(t,\,x,\,\xi) D_x^\alpha a_{j-|\alpha|}^\#(t,\,x,\,\xi) = - a(t,\,x,\,\xi) a_j^\#(t,\,x,\,\xi),
		\end{equation*}
		with
		\begin{equation*}
			a_j^\# \in \T_{2N,\,M}(0,\,0,\,-j) + \S_{N,\,M}(-j,\,0,\,0,\,0,\,0).
		\end{equation*}
		Lemma~\ref{Lemma:AsymExp} then yields the existence of a symbol
		\begin{equation*}
			a_R^\# \in \T_{2N,\,M}(0,\,0,\,0) + \S_{N,\,M}(0,\,0,\,0,\,0,\,0),
		\end{equation*}
		and a right parametrix $A_R^\#(t,\,x,\,\xi)$ with symbol $\sigma(A_R^\#) = a_R^\#$. We have
		\begin{equation*}
		A A_R^\# -I \in L^\infty\big([0,\,T];\OPS^{-\infty}\big).
		\end{equation*}
		
		The existence of a left parametrix can be shown in the same way. One can also prove that right and left parametrix coincide which yields the existence of a parametrix $A^\#$.
	\end{proof}

	Consider a pseudodifferential operator $a \in \OPS^m$ and a non-negative, increasing function $\psi\in C^\infty(\R^n)$. Throughout this paper we refer to the transformation
	\begin{equation*}
	a_{\psi}(x,\,D_x) = e^{\lambda \psi(\jbl D_x \jbr)} \circ a(x,\,D_x) \circ e^{-\lambda \psi(\jbl D_x \jbr)},
	\end{equation*}
	as conjugation, where $\lambda$ is a positive constant.
	
	\begin{Proposition}[\cite{Cicognani.2017}]\label{APP:PSEUDO:CALC:CON}
		Let $a \in \OPS^m$ and let $\psi \in C^\infty(\R^n)$ be a non-negative, increasing function satisfying
		\begin{equation}\label{APP:PSEUDO:CALC:CON:EstEta}
		\bigg|\frac{\rmd^k}{\rmd s^k}\psi (s) \bigg| \leq C_k s^{-k} \psi(s),\qquad k \in \N, \, s \in \R_+.
		\end{equation}
		We fix a constant $\lambda > 0$. Then the symbol $a_{\psi}(x,\,\xi) = \sigma(a_{\psi}(x,\,D_x))$ of
		\begin{equation*}
		a_{\psi}(x,\,D_x) = e^{\lambda \psi(\jbl D_x \jbr)} \circ a(x,\,D_x) \circ e^{-\lambda \psi(\jbl D_x \jbr)},
		\end{equation*}
		satisfies
		\begin{equation}\label{APP:PSEUDO:CALC:CON:eq1}
		a_{\psi}(x,\,\xi) = a(x,\,\xi) + \sum\limits_{0 < |\gamma| < N} a_{(\gamma)}(x,\,\xi) \chi_\gamma(\xi) + r_N(x,\,\xi),
		\end{equation}
		where
		\begin{equation}\label{APP:PSEUDO:CALC:CON:eq2}
		\chi_\gamma(\zeta) = \frac{1}{\gamma!} e^{-\lambda \psi(\jxi)} \partial_\nu^\gamma\big(e^{\lambda \psi(\jbl\nu\jbr)}\big)\Big|_{\nu = \zeta},
		\end{equation}
		and
		\begin{equation}\label{APP:PSEUDO:CALC:CON:eq3}
		\begin{aligned}
		r_N(x,\,\xi) = \frac{N}{(2\pi)^n} \sum\limits_{|\gamma| = N} \bigg[&\Osii \int\limits_0^1 (1-\vartheta)^{N-1} e^{-\I y \zeta}\\
		&\times a_{(\gamma)}(x + \vartheta y,\,\xi) \chi_\gamma(\xi + \zeta) \rmd \vartheta \rmd y \rmd\zeta\bigg].
		\end{aligned}
		\end{equation}
		Furthermore, we have the estimate
		\begin{equation}\label{APP:PSEUDO:CALC:CON:est1}
		|\partial_\xi^\alpha \chi_\gamma(\xi)| \leq C_{\alpha,\,\gamma} \jxi^{- |\alpha| - |\gamma|} (\psi(\jxi))^{|\gamma|},
		\end{equation}
		for $\xi \in \R^n$ and $\alpha \in \N^n$.
	\end{Proposition}
	\begin{Remark}
		By estimate \eqref{APP:PSEUDO:CALC:CON:est1} we are immediately able to conclude that $\chi_\gamma \in S^0$ for all $|\gamma| > 0$, if we assume $\psi(\jxi) = o(\jxi)$.
	\end{Remark}
	
	Proposition~\ref{APP:PSEUDO:CALC:CON} does not provide an estimate for $r_N(x,\,\xi)$. In order to derive such an estimate we pose additional assumptions on the operator $a$ and the function $\psi$.
	
	\begin{Proposition}[\cite{Cicognani.2017}]\label{APP:PSEUDO:CALC:CON:EstR}
		Take $a \in \OPS^m$ and $\psi \in C^\infty$ as in Proposition~\ref{APP:PSEUDO:CALC:CON}. Assume additionally that the symbol $a=a(x,\,\xi) \in \Sy^m$ is such that
		\begin{equation}\label{APP:PSEUDO:CALC:CON:EstR:EstA}
		\big|D_x^\beta \partial_\xi^\alpha a(x,\,\xi)\big| \leq C_\alpha K_{|\beta|} \jxi^{m-|\alpha|},
		\end{equation}
		for all $x,\,\xi \in \R^n$. Here $\{K_{|\beta|}\}_{|\beta|}$ is a weight sequence such that
		\begin{equation}\label{APP:PSEUDO:CALC:CON:EstR:EstK}
		\inf\limits_{p\in \N}\frac{K_{p}}{\jxi^{p}} \leq C e^{- \delta_0 \psi(\jxi)},
		\end{equation}
		for some $\delta_0 > 0$. Furthermore, we suppose that the relation
		\begin{equation}\label{APP:PSEUDO:CALC:CON:EstR:EtaAdditive}
		\psi(\jbl \xi + \zeta \jbr) \leq \psi(\jxi) + \psi(\jbl \zeta \jbr),
		\end{equation}
		holds for all large  $|\xi|,\,|\zeta|,\,\xi,\,\zeta \in \R^n$.
		We assume that the constant $\lambda > 0$ is such that there exists another positive constant $c_0$ such that
		\begin{equation}\label{APP:PSEUDO:CALC:CON:EstR:Lambda}
		\delta_0 - \lambda = c_0 > 0.
		\end{equation}
		Then the remainder $r_N(x,\,\xi)$ given by \eqref{APP:PSEUDO:CALC:CON:eq3} satisfies the estimate
		\begin{equation}\label{APP:PSEUDO:CALC:CON:EstR:EstR}
		\big| D_x^\beta \partial_\xi^\alpha r_N(x,\,\xi)\big| \leq C_{\alpha,\,\beta,\,N}  \lambda^{N}  \jxi^{m-|\alpha| - N} \psi(\jxi)^{N},
		\end{equation}
		for $(x,\,\xi) \in \R^n\times\R^n$ and $\alpha,\beta \in \N^n$.
	\end{Proposition}

	\begin{Remark}\label{APP:PSEUDO:CALC:CON:RemarkToCon}
		If we take another weight function $\widetilde \psi$ satisfying \eqref{APP:PSEUDO:CALC:CON:EstEta} and \eqref{APP:PSEUDO:CALC:CON:EstR:EtaAdditive} with $\widetilde \psi(\jxi) = o(\psi(\jxi))$, it is clear that relation \eqref{APP:PSEUDO:CALC:CON:EstR:EstK} is also satisfied. In that case, we obtain estimates \eqref{APP:PSEUDO:CALC:CON:est1} and \eqref{APP:PSEUDO:CALC:CON:EstR:EstR} for $\widetilde \psi$ without assuming \eqref{APP:PSEUDO:CALC:CON:EstR:Lambda}.
	\end{Remark}
	
	\section{Proof}\label{PROOF}
	
	The proof of Theorem~\ref{CP2:Theorem} is organized in four steps. We begin by regularizing the coefficients of the principal part and then proceed by reducing the original Cauchy problem to a Cauchy problem for a system of first order (with respect to $D_t$). It is during this reduction process, that we make use of the introduced symbol classes and zones to obtain appropriate energies in the respective parts of the extended phase space. After the reduction step, we perform a change of variables to deal with the lower order terms. Lastly, using sharp G{\aa}rding's inequality yields $L^2$-well-posedness of an related auxiliary Cauchy problem, which gives well-posedness of the original problem in the weighted spaces $H^\nu_{\eta,\,\delta}$.
	
	\subsection{Regularization}
	Since the coefficients of the principal part are just $\mu$-continuous, it is helpful to work with regularized coefficients in the hyperbolic zone.
	\begin{Definition}\label{CP2:PROOF:Def:RegCoeff}
		Let $\psi \in \Czi(\R)$ be a given function satisfying $\int_\R \psi(x) \rmd x = 1$ and $\psi(x)~\geq~0$ for any $x \in \R$  with $\supp \psi \subset [{-1},1]$. Let $\ve > 0$ and set $\psi_\ve(x) = \frac{1}{\ve} \psi\left(\frac{x}{\ve}\right)$.
		Then we define
		\begin{equation*}
			a_{\ve,\,m-j,\,\gamma}(t,\,x) := (a_{m-j,\,\gamma}\ast_t \psi_\ve)(t,\,x),
		\end{equation*}
		for $j = 0,\,\ldots,\,m-1$ and $|\gamma| + j = m$.
	\end{Definition}
	The following properties of $a_{\ve,\,m-j,\,\gamma}$ can be verified by straightforward computations.
	\begin{Proposition}[\cite{Cicognani.2017}]\label{CP2:PROOF:Prop:RegCoeff}
		The inequalities
		\begin{enumerate}[label = (\roman*),align = left, leftmargin=*]
			\item $\big|\partial_t a_{\ve,\,m-j,\,\gamma}(t,\,x)\big| \lesssim \ve^{-1} \mu(\ve)$ and
			\item $\big|D_x^\beta(a_{m-j,\,\gamma}(t,\,x) - a_{\ve,\,m-j,\,\gamma}(t,\,x))\big| \lesssim K_{|\beta|}\mu(\ve)$, for all $\beta \in \N^n$,
		\end{enumerate}
		are satisfied in $\Zhyp(N,\,M)$.
	\end{Proposition}

	\begin{Remark}
		Later on in the proof, we choose $\ve = \jxi^{-1}$ which yields the estimates
		\begin{enumerate}[label = (\roman*),align = left, leftmargin=*]
			\item $\big|\partial_t a_{\jxi^{-1},\,m-j,\,\gamma}(t,\,x)\big| \lesssim \jxi \mu(\jxi^{-1}) = C \vp(\jxi)$ and
			\item $\big|D_x^\beta\partial_\xi^\alpha \big(a_{m-j,\,\gamma}(t,\,x) - a_{\jxi^{-1},\,m-j,\,\gamma}(t,\,x)\big)\big| \lesssim K_{|\beta|} \jxi^{-|\alpha|} \mu(\jxi^{-1}) = CK_{|\beta|} \jxi^{-|\alpha|} \frac{\vp(\jxi)}{\jxi}$,
		\end{enumerate}
		for all $(t,\,x,\,\xi) \in \Zhyp(N,\,M)$.
	\end{Remark}
	
	\subsection{Reduction to a first order diagonal system}
	
	We recall that $\varrho = \varrho(t,\,\xi)$ is the positive solution to
	\begin{equation*}
	\varrho^m = 1+ \jxi \lambda(t)^m \big(w(\Lambda(t))\big)^{m(m-1)}.
	\end{equation*}
	We define the symbol
	\begin{equation*}
		h(t,\,\xi) = \varrho(t,\,\xi) \chixi + \jxi \lambda(t) \Big(1-\chixi\Big),
	\end{equation*}
	where $\chi \in C^\infty$ is a smooth cut-off function with $\chi(s) = 1$ if $|s| \leq 1$ and $\chi(s) = 0$ for $|s| \geq 2$.
	We observe that
	\begin{align*}
	\varrho(t,\,\xi) &\in \T_{N,\,M}(1,\,0,\,0)\\ &\text{ for } (t,\,x,\,\xi) \in \Zpd(N,\,M),\\
	\jxi \lambda(t) &\in \S_{N,\,M}(1,\,1,\,0,\,0,\,0)\\ &\text{ for } (t,\,x,\,\xi) \in \Zhyp(N,\,M),\\
	\varrho(t,\,\xi) \chixi  &\in \T_{2N,\,M}(1,\,0,\,0)\\ & \text{ for } (t,\,x,\,\xi) \in [0,\,T]\times\R^n\times\R^n,\\
	\jxi \lambda(t)\Big(1-\chixi\Big) &\in \S_{N,\,M}(1,\,1,\,0,\,0,\,0)\\ & \text{ for } (t,\,x,\,\xi) \in [0,\,T]\times\R^n\times\R^n,
	\end{align*}
	and thus
	\begin{equation*}
	h(t,\,\xi) \in  \T_{2N,\,M}(1,\,0,\,0)  +  \S_{N,\,M}(1,\,1,\,0,\,0,\,0).
	\end{equation*}
	Using Lemma~\ref{Lemma:rho} and Definition~\ref{SF:DEF}, we obtain that
	\begin{equation*}
	D_t h(t,\,\xi) \in \T_{2N,\,M}(1,\,1,\,0) + \S_{N,\,M}(1,\,1,\,1,\,0,\,0).
	\end{equation*}
	Next, we define the matrix pseudodifferential operator $H(t,\,D_x)$ with symbol
	\begin{equation*}
	H(t,\,\xi) = \sigma(H(t,\,D_x)) = \begin{pmatrix}
	(h(t,\,\xi))^{m-1} & & & \\[1em]
	& (h(t,\,\xi))^{m-2}& & \\[1em]
	& & \ddots&\\
	 & & & 1
	\end{pmatrix}.
	\end{equation*}
	
	\begin{Proposition}
		The inverse operator $H^{-1}(t,\,D_x)$ of $H(t,\,D_x)$ exists and its symbol is given by
		\begin{equation*}
		\sigma(H^{-1}) = (H(t,\,\xi))^{-1} =  \begin{pmatrix}
		(h(t,\,\xi))^{-(m-1)} & & & \\[1em]
		& (h(t,\,\xi))^{-(m-2)}& & \\[1em]
		& & \ddots&\\
		& & & 1
		\end{pmatrix}.
		\end{equation*}
	\end{Proposition}
	\begin{proof}
		Since $H(t,\,D_x)$ is independent of $x$, it is clear that $\sigma(H^{-1}) = (\sigma(H))^{-1}$ if $(\sigma(H))^{-1}$ exists. For this, we note that $h(t,\,\xi) \geq 1$ for all $(t,\,\xi) \in [0,\,T]\times\R^n$.
	\end{proof}

	We set
	\begin{equation*}
		U = H(t,\,D_x) \big(u(t,\,x),\,D_t u(t,\,x),\,\ldots,\,D_t^{m-1} u(t,\,x)\big)^T.
	\end{equation*}
	Applying this transformation to the Cauchy problem \eqref{CP2:CP} yields
	\begin{equation}\label{CP2:PROOF:TransformedSystem}
	D_t U = (A+B) U,
	\end{equation}
	with initial conditions
	\begin{equation*}
		U(0,\,x) = H(0,\,D_x) \big(u(0,\,x),\,D_t u(0,\,x),\,\ldots,\,D_t^{m-1} u(0,\,x)\big)^T,
	\end{equation*}
	where
	\begin{equation*}
		\sigma(A) = \begin{pmatrix}
		 & h(t,\,\xi) &&\\
		 & & \ddots& \\
		 & & & h(t,\,\xi)\\
		 a_m(t,\,x,\,\xi) & a_{m-1}(t,\,x,\,\xi) & \ldots &a_{1}(t,\,x,\,\xi)
		\end{pmatrix},
	\end{equation*}
	and
	\begin{equation*}
		\sigma(B) = \begin{pmatrix}
		\frac{(m-1)D_t(h(t,\,\xi))}{h(t,\,\xi)} & & &\\
		& \ddots& & &\\
		& & \frac{D_t(h(t,\,\xi))}{h(t,\,\xi)}&\\
		b_m(t,\,x,\,\xi) & b_{m-1}(t,\,x,\,\xi) & \ldots & b_{1}(t,\,x,\,\xi)
		\end{pmatrix},
	\end{equation*}
	with
	\begin{align*}
		a_{m-k}(t,\,x,\,\xi) &=
		\begin{aligned}[t] &\chixi \sum\limits_{|\gamma| = m-k} \frac{\lambda(t)^{m-k} a_{m-k,\,\gamma}(t,\,x) \xi^\gamma}{ h(t,\,\xi)^{m-1-k}}\\
		&+ \Big(1-\chixi\Big)\\&\times\sum\limits_{|\gamma| = m-k} \frac{\lambda(t)^{m-k} a_{\ve,\,m-k,\,\gamma}(t,\,x) \xi^\gamma}{ h(t,\,\xi)^{m-1-k}},
		\end{aligned}\\
		b_{m-k}(t,\,x,\,\xi) &=
		\begin{aligned}[t] & \sum\limits_{|\gamma| < m-k} \frac{b_{m-k,\,\gamma}(t,\,x) \xi^\gamma}{ h(t,\,\xi)^{m-1-k}}
		+ \Big(1-\chixi\Big)\\ &\times\sum\limits_{|\gamma| = m-k} \frac{\lambda(t)^{m-k} (a_{m-k,\,\gamma}(t,\,x)-a_{\ve,\,m-k,\,\gamma}(t,\,x)) \xi^\gamma}{ h(t,\,\xi)^{m-1-k}},
		\end{aligned}
	\end{align*}
	and $a_{\ve,\,m-k,\,\gamma}(t,\,x)$ are the regularized coefficients of Definition~\ref{CP2:PROOF:Def:RegCoeff} with $\ve = \jxi^{-1}$ and $k = 0,\,\ldots,\,m-1$.
	
	By using this approach, in the pseudodifferential zone the pseudodifferential operator $A$ contains all coefficients of the principal part of the original equation. In the hyperbolic zone, the original coefficients are replaced by the regularized coefficients (which we need for the diagonalization in the hyperbolic zone). The pseudodifferential operator $B$ contains all lower order terms and the terms arising due to the regularization in the hyperbolic zone.
	
	\begin{Proposition}
		The symbol $a_{m-k} = a_{m-k}(t,\,x,\,\xi)$ belongs to
		\begin{equation*}
		\T_{2N,\,M}(1,\,0,\,0) + \S_{N,\,M}(1,\,1,\,0,\,0,\,0),
		\end{equation*}
		the symbol $b_{m-k}=b_{m-k}(t,\,x,\,\xi)$ belongs to
		\begin{equation*}
		\T_{2N,\,M}(1,\,0,\,0) + \S_{N,\,M}(0,\,1,\,0,\,1,\,0) + \S_{N,\,M}(0,\,1,\,0,\,0,\,1).
		\end{equation*}
		Furthermore, they all belong to $B^\infty_K$ with respect to the spatial variables.
	\end{Proposition}
	\begin{proof}
		We first observe that
		\begin{equation*}
		\begin{aligned}
			h(t,\,\xi)^{-(m-1-k)} &\in \T_{2N,\,M}(-(m-1-k),\,0,\,0)\\& + \S_{N,\,M}(-(m-1-k),\,-(m-1-k),\,0,\,0,\,0).
			\end{aligned}
		\end{equation*}
		Application of Lemma~\ref{Lemma:SymbolProd} yields the desired symbol class for $a_{m-k}$, when we keep in mind that all $a_{m-k,\,\gamma}$ and $a_{\ve,\,m-k,\,\gamma}$ are bounded on $[0,\,T]$ and
		\begin{equation*}
			\chixi \sum\limits_{|\gamma| = m-k} \lambda(t)^{m-k} \xi^\gamma \in \T_{2N,\,M}(m-k,\,0,\,0),
		\end{equation*}
		where we used that $\jxi^{m-k} \leq C \big(w(\Lambda(t))\big)^{(m-1)(m-k)} \jxi^{\frac{m-k}{m}}$ in $\Zpd(2N,\,M)$.
		
		The first part of the symbol class for $b_{m-k}$ is obtained by using the Levi-condition \ref{CP2:LeviCond}, i.e.
		\begin{equation*}
			\big|D_x^\beta b_{m-k,\,\gamma}(t,\,x)\big| \leq C K_{|\beta|} \lambda(t)^{m-k} \big(w(\Lambda(t))\big)^{m(m-k-|\gamma|)}.
		\end{equation*}
		We obtain that
		\begin{equation*}
			\sum\limits_{|\gamma| < m-k} b_{m-k,\,\gamma}(t,\,x) \xi^\gamma \in \T_{2N,\,M}(m-k,\,0,\,0) + \S_{N,\,M}(m-k-1,\,m-k,\,0,\,1,\,0),
		\end{equation*}
		where we used that $\jxi^{|\gamma|} \leq C \big(w(\Lambda(t))\big)^{m|\gamma|-(m-k)} \jxi^{\frac{m-k}{m}}$ in $\Zpd(2N,\,M)$, as well as $\big(w(\Lambda(t))\big)^{m(m-k-|\gamma|)} \jxi^{|\gamma|} \leq C \jxi^{m-k-1} \big(w(\Lambda(t))\big)^m$ in $\Zhyp(N,\,M)$.
		This yields that
		\begin{equation*}
			\sum\limits_{|\gamma| < m-k} \frac{b_{m-k,\,\gamma}(t,\,x) \xi^\gamma}{ h(t,\,\xi)^{m-1-k}}
			\in \T_{N,\,M}(1,\,0,\,0) + \S_{N,\,M}(0,\,1,\,0,\,1,\,0).
		\end{equation*}
	
		For the second summand of $b_{m-k}(t,\,x,\,\xi)$, we note that in view of Proposition~\ref{CP2:PROOF:Prop:RegCoeff} we have
		\begin{align*}
			\Big(1-\chixi\Big)\sum\limits_{|\gamma| = m-k} \lambda(t)^{m-k} \big(a_{m-k,\,\gamma}(t,\,x)-a_{\ve,\,m-k,\,\gamma}(t,\,x)\big) \xi^\gamma\\ \in \S_{N,\,M}(m-k-1,\,m-k,\,0,\,0,\,1),
		\end{align*}
		which yields that
		\begin{align*}
			\Big(1-\chixi\Big)\sum\limits_{|\gamma| = m-k} \frac{\lambda(t)^{m-k} (a_{m-k,\,\gamma}(t,\,x)-a_{\ve,\,m-k,\,\gamma}(t,\,x)) \xi^\gamma}{ h(t,\,\xi)^{m-1-k}}\\ \in \S_{N,\,M}(0,\,1,\,0,\,0,\,1).
		\end{align*}
		This gives the desired symbol classes.
		
		The fact that $a_{m-k}=a_{m-k}(t,\,x,\,\xi)$ and $b_{m-k}=b_{m-k}(t,\,x,\,\xi)$ are all $B^\infty_K$ in $x$ follows from \ref{CP2:CoeffPrinc} and \ref{CP2:LeviCond} and the symbolic calculus, where we include $K_{|\beta|}$ in all estimates.
	\end{proof}
	We conclude that
	\begin{align*}
		\sigma(A) \in \T_{2N,\,M}(1,\,0,\,0) &+ \S_{N,\,M}(1,\,1,\,0,\,0,\,0),\\
		\sigma(B) \in \T_{2N,\,M}(1,\,0,\,0)&+ \T_{2N,\,M}(0,\,1,\,0) + \S_{N,\,M}(0,\,1,\,0,\,1,\,0)\\ & + \S_{N,\,M}(0,\,1,\,0,\,0,\,1)+ \S_{N,\,M}(0,\,0,\,1,\,0,\,0),
	\end{align*}
	where we used that
	\begin{equation*}
		(D_t h(t,\,\xi)) h(t,\,\xi)^{-1} \in \T_{2N,\,M}(0,\,1,\,0)+ \S_{N,\,M}(0,\,0,\,1,\,0,\,0).
	\end{equation*}
	
	To prepare diagonalization in the hyperbolic zone, we introduce approximated characteristic roots $\tau_k = \tau_k(t,\,x,\,\xi)$ which are the solutions to
	\begin{equation*}
		\tau^m = \sum\limits_{j = 0}^{m-1} \sum\limits_{|\gamma| = m-j} \lambda(t)^{m-j}a_{\ve,\,m-j,\,\gamma}(t,\,x)\xi^\gamma \tau^j .
	\end{equation*}
	
	\begin{Proposition}
		For $\ve = \jxi^{-1}$, the roots $\tau_k = \tau_k(t,\,x,\,\xi)$ satisfy
		\begin{enumerate}[label = (\roman*)]
			\item $\big|D_x^\beta \partial_\xi^\alpha\tau_k\big|\leq C_\alpha K_{|\beta|} \jxi^{1-|\alpha|} \lambda(t)$, for all $k = 1,\,\ldots,\,m$,
			\item $	|\partial_t \tau_k| \leq C\lambda(t) \jxi \Big(\frac{\lambda^\prime(t)}{\lambda(t)} + \vp(\jxi)\Big)$,	for $k = 1,\,\ldots,\,m$ and all $(t,\,x,\,\xi) \in \Zhyp(N,\,M)$, i.e. $\partial_t \tau_k \in \S_{N,\,M}(1,\,1,\,1,\,0,\,0) + \S_{N,\,M}(1,\,1,\,0,\,0,\,1)$.
		\end{enumerate}
	\end{Proposition}
	\begin{proof}
	We only show the proof of the second assertion since the first follows from the definition of $\tau_k$ and the fact that the Cauchy problem is strictly hyperbolic in $\Zhyp(N,\,M)$.
		
	For (ii), we apply the implicit function theorem to
	\begin{equation*}
		P(\tau(t,\,x,\,\xi),\,t,\,x,\,\xi) = \tau^m - \sum\limits_{j=0}^{m-1} \sum\limits_{|\gamma| = m-j} a_{\ve,\,m-j,\,\gamma}(t,\,x) \xi^\gamma \tau^j = 0,
	\end{equation*}
	and obtain that
	\begin{equation*}
		\partial_t \tau(t,\,x,\,\xi) = - \frac{P_t}{P_\tau} =  \frac{-\sum\limits_{j=0}^{m-1} \sum\limits_{|\gamma|=m-j} (\partial_t (\lambda(t)^{m-j} a_{\ve,\,m-j,\,\gamma}(t,\,x))) \xi^\gamma \tau^j}{m \tau^{m-1} - \sum\limits_{j = 1}^{m-1}\sum\limits_{|\gamma|=m-j} j \lambda(t)^{m-j}a_{\ve,\,m-j,\,\gamma}(t,\,x) \xi^\gamma \tau^{j-1}}.
	\end{equation*}
	We use that $|\tau| \sim \lambda(t) \jxi$ in $\Zhyp(N,\,M)$ to estimate
	\begin{align*}
		|\partial_t \tau| &\lesssim \frac{\sum\limits_{j=0}^{m-1} \sum\limits_{|\gamma|=m-j} \lambda(t)^{m} \jxi^{|\gamma|+j}\Big(j\frac{\lambda^\prime(t)}{\lambda(t)} a_{\ve,\,m-j,\,\gamma}(t,\,x) + \big|\partial_t a_{\ve,\,m-j,\,\gamma}(t,\,x)\big|\Big)  }{\lambda(t)^{m-1} \jxi^{m-1}}\\
		&\lesssim \lambda(t) \jxi \Big(\frac{\lambda^\prime(t)}{\lambda(t)} + \vp(\jxi)\Big),
	\end{align*}
	where we used Proposition~\ref{CP2:PROOF:Prop:RegCoeff} to estimate $|\partial_t a_{\ve,\,m-j,\,\gamma}(t,x)|$.
	\end{proof}

	We renumber the roots $\tau_k$ in such a way that
	\begin{equation*}
		\tau_1 < \tau_2 < \ldots < \tau_m
	\end{equation*}
	for all $(t,\,x,\,\xi) \in \Zhyp(N,\,M)$ and define for $k =  1,\,\ldots,\,m$ the symbols
	\begin{equation*}
		\psi_k(t,\,x,\,\xi) = d_k \varrho(t,\,\xi) \chixi + \tau_k(t,\,x,\,\xi) \Big(1-\chixi\Big),
	\end{equation*}
	where $d_1 < d_2 < \ldots < d_m$ are real, positive numbers.
	
	\begin{Proposition}\label{CP2:PROOF:psi}
		The symbols $\psi_k=\psi_k(t,\,x,\,\xi)$ and $h=h(t,\,\xi)$ satisfy
		\begin{enumerate}[label = (\roman*)]
			\item $\psi_k\in \T_{2N,\,M}(1,\,0,\,0) + \S_{N,\,M}(1,\,1,\,0,\,0,\,0)$,
			\item $\partial_t \psi_k \in \T_{2N,\,M}(1,\,1,\,0)+ \S_{N,\,M}(1,\,1,\,1,\,0,\,0)+ \S_{N,\,M}(1,\,1,\,0,\,0,\,1)$,
			\item $\sigma(\psi_k \circ h^{-1})\in \T_{2N,\,M}(0,\,0,\,0)+\S_{N,\,M}(0,\,0,\,0,\,0,\,0),$
			\item $\sigma(\partial_t(\psi_k \circ h^{-1})) \in \S_{N,\,M}(0,\,0,\,1,\,0,\,0) + \S_{N,\,M}(0,\,0,\,0,\,0,\,1).$
		\end{enumerate}
	\end{Proposition}
	\begin{proof}
		Assertion (i) follows immediately from Lemma~\ref{Lemma:SymbolProd}.
		
		As for (ii), we note that
		\begin{align*}
		\partial_t \psi_k &= \chi^\prime\Big(\frac{\jxi}{N(w(\Lambda(t)))^m}\Big) \frac{\lambda(t)}{\Lambda(t)} \varrho + \chixi \partial_t \varrho\\ &+ \chi^\prime\Big(\frac{\jxi}{N(w(\Lambda(t)))^m}\Big) \frac{\lambda(t)}{\Lambda(t)} \tau_k + \Big(1-\chixi\Big) \partial_t \tau_k.
		\end{align*}
		Using Lemma~\ref{Lemma:SymbolProd} and relation \eqref{Lemma:rho:estDtRho} for $\partial_t \varrho$ yields the desired symbol class.
		
		The third assertion is a consequence of
		\begin{align*}
		\sigma(\psi_k \circ h^{-1}) &\sim \sum\limits_{\alpha \in \N^n} \partial_\xi^\alpha \psi_k (t,\,x,\,\xi) D_x^\alpha h^{-1}(t,\,\xi)= \frac{\psi_k(t,\,x,\,\xi)}{h(t,\,\xi)}\\ &= \frac{d_k \chixi \varrho(t,\,\xi) + \Big(1-\chixi\Big) \tau_k(t,\,x,\,\xi)}{\chixi \varrho(t,\,\xi) + \Big(1-\chixi\Big)\lambda(t) \jxi},
		\end{align*}
		while the last assertion follows from
		\begin{equation*}
		\partial_t \tau_k \in \S_{N,\,M}(1,\,1,\,1,\,0,\,0) + \S_{N,\,M}(1,\,1,\,0,\,0,\,1).
		\end{equation*}
	\end{proof}
		
	We define the pseudodifferential operator $M = M(t,\,x,\,D_x)$ with symbol
	\begin{equation*}
		M(t,\,x,\,\xi) = \sigma(M) = \begin{pmatrix}
			1 & \ldots & 1\\
			\frac{\psi_1(t,\,x,\,\xi)}{h(t,\,\xi)} & \ldots & \frac{\psi_m(t,\,x,\,\xi)}{h(t,\,\xi)}\\
			\vdots & \vdots & \vdots\\
			\Big(\frac{\psi_1(t,\,x,\,\xi)}{h(t,\,\xi)}\Big)^{m-1} & \ldots & \Big(\frac{\psi_m(t,\,x,\,\xi)}{h(t,\,\xi)}\Big)^{m-1}
		\end{pmatrix},
	\end{equation*}
	which belongs to $\T_{2N,\,M}(0,\,0,\,0)+\S_{N,\,M}(0,\,0,\,0,\,0,\,0)$. By construction the symbols $\psi_k(t,\,x,\,\xi)$ satisfy $\psi_k(t,\,x,\,\xi) \neq \psi_j(t,\,x,\,\xi)$ if $k \neq j$ for all $(t,\,x,\,\xi) \in [0,\,T]\times\R^n\times\R^n$. Hence, $\det M(t,\,x,\,\xi) \neq 0$ and the matrix $M(t,\,x,\,\xi)$ is invertible, which allows us to apply Lemma~\ref{Lemma:Parametrix} guaranteeing the existence of a parametrix $M^\#(t,\,x,\,D_x)$ with symbol
	\begin{equation*}
	M^\#(t,\,x,\,\xi) = \sigma(M^\#) \in\T_{2N,\,M}(0,\,0,\,0)+\S_{N,\,M}(0,\,0,\,0,\,0,\,0).
	\end{equation*}
	
	\medskip
	
	We define $V$ by $U = M V$ and obtain that Cauchy problem \eqref{CP2:PROOF:TransformedSystem} can be transformed to
	\begin{equation}\label{CP2:PROOF:TransformedSystem2}
		M^\# M D_t V = (M^\# \circ (A+B) \circ M)V - M^\#(D_t M) V,
	\end{equation}
	with initial conditions
	\begin{equation*}
		V(0,\,x) = M^\#(0,\,x,\,D_x) H(0,\,D_x) \big(u(0,\,x),\,D_t u(0,\,x),\,\ldots,\,D_t^{m-1} u(0,\,x)\big)^T.
	\end{equation*}
	Using the composition results from Lemma~\ref{Lemma:CompPseu}, we obtain that
	\begin{equation*}
		\sigma(M^\#\circ A \circ M) = \sigma(M^\#) \sigma(A) \sigma(M) + f_0 + r_\infty,
	\end{equation*}
	with $f_0 = 0$ in $\Zpd(N,\,M)$ and $f_0 \in \T_{2N,\,M}(1,\,0,\,0) + \S_{N,\,M}(0,\,1,\,0,\,0,\,0)$ in the remaining part of the extended phase space, and $r_\infty \in C^\infty\big([0,\,T];\,\OPS^{-\infty}\big)$.
	Due to the construction of $M$ and $M^\#$, we have that
	\begin{equation*}
		\sigma(M^\#) \sigma(A) \sigma(M)
			=\begin{pmatrix}
				\tau_1(t,\,x,\,\xi) & &\\
				& \ddots & \\
				& & \tau_m(t,\,x,\,\xi)
			\end{pmatrix} \text{ in } \Zhyp(2N,\,M),
	\end{equation*}
	and $\sigma(M^\#) \sigma(A) \sigma(M) \in \T_{2N,\,M}(1,\,0,\,0)$ in the remaining part of the extended phase space. Applying Lemma~\ref{Lemma:CompPseu} to $M^\#\circ B \circ M$ yields
	\begin{align*}
		\sigma(M^\#\circ B \circ M) \in \T_{2N,\,M}(1,\,0,\,0)+ \T_{2N,\,M}(0,\,1,\,0) + \S_{N,\,M}(0,\,1,\,0,\,1,\,0)\\  + \S_{N,\,M}(0,\,1,\,0,\,0,\,1)+ \S_{N,\,M}(0,\,0,\,1,\,0,\,0).
	\end{align*}
	Lastly, we consider the term $M^\#(D_t M)$.
	To characterize $D_t M$, we only have to consider the symbols
	\begin{equation*}
		\partial_t \Big(\frac{\psi_k(t,\,x,\,\xi)}{h(t,\,\xi)}\Big)^{j} =  j \Big(\frac{\psi_k(t,\,x,\,\xi)}{h(t,\,\xi)}\Big)^{j-1} \partial_t\Big(\frac{\psi_k(t,\,x,\,\xi)}{h(t,\,\xi)}\Big),
	\end{equation*}
	for $j = 0,\,\ldots,\,m-1$. Since $\frac{\psi_k(t,\,x,\,\xi)}{h(t,\,\xi)}\in \T_{2N,\,M}(0,\,0,\,0) + \S_{N,\,M}(0,\,0,\,0,\,0,\,0)$, we just consider $\partial_t \Big(\frac{\psi_k(t,\,x,\,\xi)}{h(t,\,\xi)}\Big)$ and obtain that
	\begin{equation*}
	\partial_t \Big(\frac{\psi_k(t,\,x,\,\xi)}{h(t,\,\xi)}\Big) = \frac{(\partial_t \psi_k(t,\,x,\,\xi))h(t,\,\xi) - \psi_k(t,\,x,\,\xi) (\partial_t h(t,\,\xi))}{h(t,\,\xi)^2}.
	\end{equation*}
	This yields that
	\begin{equation*}
	\partial_t \Big(\frac{\psi_k(t,\,x,\,\xi)}{h(t,\,\xi)}\Big) = 0 \text{ in } \Zpd(N,\,M),
	\end{equation*}
	and \begin{equation*}
	\partial_t \Big(\frac{\psi_k(t,\,x,\,\xi)}{h(t,\,\xi)}\Big) \in \S_{N,\,M}(0,\,0,\,1,\,0,\,0) + \S_{N,\,M}(0,\,0,\,0,\,0,\,1),
	\end{equation*}
	in the remaining part of the extended phase space. From this we conclude that
	\begin{align*}
	\sigma(M^\# (D_tM)) &= 0 \text{ in } \Zpd(N,\,M), \text{ and }\\
	\sigma(M^\# (D_tM)) &\in \S_{N,\,M}(0,\,0,\,1,\,0,\,0) + \S_{N,\,M}(0,\,0,\,0,\,0,\,1),
	\end{align*}
	since $\sigma(M^\#) \in \T_{2N,\,M}(0,\,0,\,0) + \S_{N,\,M}(0,\,0,\,0,\,0,\,0)$.
	
	We rewrite Cauchy problem \eqref{CP2:PROOF:TransformedSystem2} as
	\begin{equation}\label{CP2:PROOF:TransformedSystem3}
	D_t V = D(t,\,x,\,D_x) V + R(t,\,x,\,D_x) V,
	\end{equation}
	with initial conditions
	\begin{equation*}
		V(0,\,x) = M^\#(0,\,x,\,D_x) H(0,\,D_x) \big(u(0,\,x),\,D_t u(0,\,x),\,\ldots,\,D_t^{m-1} u(0,\,x)\big)^T,
	\end{equation*}
	where
	\begin{align*}
	\sigma(D(t,\,x,\,D_x)) &= \Big(1-\chixi\Big)	\begin{pmatrix}
						\tau_1 & &\\
						& \ddots & \\
						& & \tau_m
						\end{pmatrix}\\& \in \T_{2N,\,M}(0,\,0,\,-\infty) + \S_{N,\,M}(1,\,1,\,0,\,0,\,0),
	\end{align*}
	and
	\begin{align}
	\nonumber
	\sigma(R(t,\,x,\,D_x)) \in \T_{2N,\,M}(1,\,0,\,0)&+ \T_{2N,\,M}(0,\,1,\,0) + \S_{N,\,M}(0,\,1,\,0,\,1,\,0)\\\nonumber & + \S_{N,\,M}(0,\,1,\,0,\,0,\,1)+ \S_{N,\,M}(0,\,0,\,1,\,0,\,0)\\&+ \S_{N,\,M}(0,\,1,\,0,\,0,\,0) + \S_{N,\,M}(0,\,0,\,0,\,0,\,1).\label{PROOF:EstR}
	\end{align}

	\subsection{Conjugation}
	To obtain our desired energy estimate, we want to control the lower order terms in $R(t,\,x,\,D_x)$ by applying a change of variable that contains the loss of derivatives. For this purpose, we introduce for $t_0 \in[0,\,T]$ the pseudodifferential operator $\Phi=\Phi(t_0,\,D_x)$ having the symbol
	\begin{align*}
	&\sigma(\Phi(t_0,\,D_x))=\\ & \widetilde M_1 \int\limits_0^{t_0} \varrho( t,\,\xi) \chixi \rmd t + \widetilde M_2 \int\limits_0^{t_0} \frac{\partial_{t} \varrho(t,\,\xi)}{\varrho(t,\,\xi)} \chixi \rmd t
	\\
	&+ \widetilde M_3 \int\limits_0^{t_0} \lambda(t) \big(w(\Lambda( t))\big)^{m} \widetilde \chi\Big(\frac{\jxi}{N (w(\Lambda(t)))^m}\Big) \rmd  t \\ &+ \widetilde M_4 \int\limits_0^{t_0} \lambda( t) \vp(\jxi) \widetilde \chi\Big(\frac{\jxi}{N (w(\Lambda(t)))^m}\Big) \rmd  t +  \widetilde M_5 \int\limits_0^{t_0} \frac{\lambda( t)}{\Lambda( t)}\widetilde \chi\Big(\frac{\jxi}{N (w(\Lambda(t)))^m}\Big) \rmd  t\\& + \widetilde M_6 \int\limits_0^{t_0} \lambda( t) \widetilde \chi\Big(\frac{\jxi}{N (w(\Lambda(t)))^m}\Big) \rmd t \\
	&+ \widetilde M_7 \int\limits_0^{t_0}\vp(\jxi) \widetilde \chi\Big(\frac{\jxi}{N (w(\Lambda(t)))^m}\Big)\rmd t - M_8(T-\kappa t_0) \jxi^{\frac{1}{s}},
	\end{align*}
	where
	\begin{equation*}
		\widetilde \chi\Big(\frac{\jxi}{N (w(\Lambda(t)))^m}\Big) = \Big(1-\chixi\Big),
	\end{equation*}
and $\kappa > 0$ is determined later.
	
	We set
	\begin{equation*}
		V = \jbl D_x \jbr^{-\nu}  e^{\Phi(t,\,D_x)} W,
	\end{equation*}
	and obtain that Cauchy problem \eqref{CP2:PROOF:TransformedSystem3} is equivalent to
	\begin{align}
	\nonumber
			D_t W &=
			\begin{aligned}[t]
			&e^{-\Phi(t,\,D_x)} \jbl D_x \jbr^{\nu} D(t,\,x,\,D_x) \jbl D_x \jbr^{-\nu} e^{\Phi(t,\,D_x)}W\\&+ e^{-\Phi(t,\,D_x)}\jbl D_x \jbr^{\nu} R(t,\,x,\,D_x)\jbl D_x \jbr^{-\nu}  e^{\Phi(t,\,D_x)}W\\ &- e^{-\Phi(t,\,D_x)} \Big( D_t e^{\Phi(t,\,D_x)}\Big) W
			\end{aligned}\\
			\label{CP2:PROOF:TransformedSystem4}
			&=\begin{aligned}[t]&e^{-\Phi(t,\,D_x)} D(t,\,x,\,D_x) e^{\Phi(t,\,D_x)}W + e^{-\Phi(t,\,D_x)}R(t,\,x,\,D_x) e^{\Phi(t,\,D_x)}W\\ &- \Big( \sum_{k = 1}^8 D_t \Phi_k(t,\,D_x)\Big) W + R_\infty(t,\,x,\,D_x) W,
			\end{aligned}
	\end{align}
	with initial conditions
	\begin{equation*}
	\begin{aligned}
		W(0,\,x) =  \jbl D_x \jbr^{\nu} e^{-\Phi(0,\,D_x)} M^\#(0,\,x,\,D_x) H(0,\,D_x)\\\times \big(u(0,\,x),\,\ldots,\,D_t^{m-1}u(0,\,x)\big)^T,
	\end{aligned}
	\end{equation*}
	where $\Phi_k,\,k=1,\,\ldots,\,8$, denote the respective addends of $\Phi$ in order of their appearance and $R_\infty \in C\big([0,\,T];\, \OPS^{-\infty}\big)$.
	We want to apply Propositions~\ref{APP:PSEUDO:CALC:CON} and \ref{APP:PSEUDO:CALC:CON:EstR} to evaluate the conjugations $D_\Phi = e^{-\Phi} D e^{\Phi}$ and $R_\Phi = e^{-\Phi} R e^\Phi$. To this end, we note that
	\begin{align*}
		\widetilde M_2 \int\limits_0^{t_0} \frac{\partial_{t} \varrho(t,\,\xi)}{\varrho(t,\,\xi)} \chixi \rmd t &\leq C \widetilde M_2 \log(\varrho(t_0,\,\xi)) \chi\Big(\frac{\jxi}{N (w(\Lambda(t_0)))^m}\Big)\\ &\leq C \widetilde M_2 \log(\jxi) \chi\Big(\frac{\jxi}{N (w(\Lambda(t_0)))^m}\Big),
	\end{align*}
	and
	\begin{equation*}
	 \widetilde M_6 \int\limits_0^{t_0} \lambda( t) \widetilde \chi\Big(\frac{\jxi}{N (w(\Lambda(t)))^m}\Big) \rmd t  \leq C,
	\end{equation*}
	as well as
	\begin{equation*}
	\begin{aligned}
		 \widetilde M_5 &\int\limits_0^{t_0} \frac{\lambda( t)}{\Lambda( t)}\widetilde \chi\Big(\frac{\jxi}{N (w(\Lambda(t)))^m}\Big) \rmd  t\\ &\leq C \widetilde M_5  \int\limits_0^{t_1} \lambda(t) \big(w(\Lambda( t))\big)^{m} \widetilde \chi\Big(\frac{\jxi}{N (w(\Lambda(t)))^m}\Big) \rmd  t + C,
	\end{aligned}
	\end{equation*}
	due to Remark~\ref{CP2:Remark:EstLambdaW}, where $0 < t_1 $ is sufficiently small. Thus, we may write
	\begin{equation}\label{CP2:PROOF:EstPhi}
		\begin{aligned}[t]
		&\sigma\big(e^{\Phi(t_0,\,D_x)}\big)=\\
		&\exp\bigg\{M_1 \int\limits_0^{t_0} \varrho( t,\,\xi) \chixi \rmd t +  M_2 \log(\jxi) \chi\Big(\frac{\jxi}{N (w(\Lambda(t_0)))^m}\Big)
		\\
		&+ M_3 \int\limits_0^{t_0} \lambda(t) \big(w(\Lambda( t))\big)^{m} \widetilde \chi\Big(\frac{\jxi}{N (w(\Lambda(t)))^m}\Big) \rmd  t +  M_4 \vp(\jxi) \widetilde \chi\Big(\frac{\jxi}{N (w(\Lambda(t)))^m}\Big)\\
		&+ M_8(T-\kappa t_0) \jxi^{\frac{1}{s}}\bigg\}.
		\end{aligned}
	\end{equation}
	Here we observe, that
	\begin{align*}
		\int\limits_0^{t_0}& \varrho( t,\,\xi) \chixi \rmd t\\ &\leq C \int\limits_0^{t_0} \jxi^\frac{1}{m} \lambda(t) \big(w(\Lambda(t))\big)^{m-1} \chixi \rmd t + C\\ &\leq C W(\Lambda(t_\xi)) \big(w(\Lambda(t_\xi))\big)^{m-1}+C,
	\end{align*}
	due to the definition of $\varrho$, the definitions of the zones and the special choice of $w(\Lambda(t))$. Similarly, we obtain
	\begin{equation*}
		 \int\limits_0^{t_0} \lambda(t) \big(w(\Lambda( t))\big)^{m} \widetilde \chi\Big(\frac{\jxi}{N (w(\Lambda(t)))^m}\Big) \rmd  t \leq C W(\Lambda(t_\xi)) \big(w(\Lambda(t_\xi))\big)^{m-1}+C.
	\end{equation*}
	We now use assumption \ref{CP2:EstEta} to conclude that $\Phi$ satisfies the assumptions of Propositions~\ref{APP:PSEUDO:CALC:CON} and \ref{APP:PSEUDO:CALC:CON:EstR} and obtain that
	\begin{align*}
	\sigma(D_\Phi(t,\,x,\,D_x)) &= D(t,\,x,\,\xi) + \sum\limits_{0 < |\gamma| < N} D_x^\gamma D(t,\,x,\,\xi) \chi_\gamma(\xi) + r_N( D;\,t,\,x,\,\xi),\\
	\sigma(R_\Phi(t,\,x,\,D_x)) &= R(t,\,x,\,\xi) + \sum\limits_{0 < |\gamma| < N} D_x^\gamma R(t,\,x,\,\xi) \chi_\gamma(\xi) + r_N( R;\,t,\,x,\,\xi),
	\end{align*}
	where
	\begin{equation*}
	|\partial_\xi^\alpha \chi_\gamma(\xi)| \leq C_{\alpha,\,\gamma} \jxi^{-|\alpha|-|\gamma|}(\Phi(t,\,\xi))^{|\gamma|},
	\end{equation*}
	\begin{align*}
	\big|D_x^\beta \partial_\xi^\alpha r_N( D;\,t,\,x,\,\xi)\big| &\leq C_{\alpha,\,\beta,\,N} \jxi^{1-|\alpha|} \Big(\frac{\Phi(t,\xi)}{\jxi}\Big)^N,\text{ and }\\
	\big|D_x^\beta \partial_\xi^\alpha r_N(R;\,t,\,x,\,\xi)\big| &\leq C_{\alpha,\,\beta,\,N} \jxi^{1-|\alpha|} \Big(\frac{\Phi(t,\,\xi)}{\jxi}\Big)^N,
	\end{align*}
	for $t \in [0,\,T]$.
	Thus,
	\begin{align*}
	\sigma(D_\Phi(t,\,x,\,D_x)) &= D(t,\,x,\,\xi) + \sum\limits_{0 < |\gamma| \leq N} R_{\gamma}(D;\,t,\,x,\,\xi),\\
	\sigma(R_\Phi(t,\,x,\,D_x)) &= R(t,\,x,\,\xi) + \sum\limits_{0 < |\gamma| \leq N} R_{\gamma}(R;\,t,\,x,\,\xi),
	\end{align*}
	where
	\begin{align*}
	\big|D_x^\beta \partial_\xi^\alpha R_\gamma( D;\,t,\,x,\,\xi)\big| &\leq C_{\alpha,\,\beta,\,|\gamma|} \jxi^{1-|\alpha|} \Big(\frac{\Phi(t,\,\xi)}{\jxi}\Big)^{|\gamma|}\\&\leq C_{\alpha,\,\beta,\,|\gamma|} \jxi^{-|\alpha|} \Phi(t,\,\xi),\\
	\big|D_x^\beta \partial_\xi^\alpha R_\gamma(R;\,t,\,x,\,\xi)\big| &\leq C_{\alpha,\,\beta,\,|\gamma|} \jxi^{1-|\alpha|} \Big(\frac{\Phi(t,\,\xi)}{\jxi}\Big)^{|\gamma|}\\&\leq C_{\alpha,\,\beta,\,|\gamma|} \jxi^{-|\alpha|} \Phi(t,\,\xi).
	\end{align*}
	Most importantly, due to Remark~\ref{APP:PSEUDO:CALC:CON:RemarkToCon} all these relations are satisfied with constants independent of $t$ or $T$.
	We conclude that Cauchy problem \eqref{CP2:PROOF:TransformedSystem3} is equivalent to
	\begin{equation*}
	\begin{aligned}
	D_t W = D(t,\,x,\,D_x) W + R(t,\,x,\,D_x)W + R_1(t,\,x,\,D_x)W\\- \Big( \sum_{k = 1}^8 D_t \Phi_k(t,\,D_x)\Big) W,
	\end{aligned}
	\end{equation*}
	with initial conditions
	\begin{equation*}
	\begin{aligned}
	W(0,\,x) =  \jbl D_x \jbr^{\nu} e^{-\Phi(0,\,D_x)} M^\#(0,\,x,\,D_x) H(0,\,D_x)\\\times \big(u(0,\,x),\,\ldots,\,D_t^{m-1}u(0,\,x)\big)^T,
	\end{aligned}
	\end{equation*}
	where
	\begin{equation}\label{PROOF:EstR1}
	\big|D_x^\beta \partial_\xi^\alpha \sigma(R_1(t,\,x,\,D_x))\big| \leq C_{\alpha,\,\beta} \jxi^{-|\alpha|} \Phi(t,\,\xi),
	\end{equation}
	for $t \in [0,\,T]$ and $D(t,\,x,\,D_x)$ and $R(t,\,x,\,D_x)$ as before.

	\subsection{Well-posedness of an auxiliary Cauchy problem}
	
	In this section we consider the auxiliary Cauchy problem
	\begin{equation}\label{SH:PROOF:AUXCP:WEAK:CP}
	\begin{aligned}
	\partial_t W =	 \Big(\I  D(t,\,x,\,D_x) + \I R(t,\,x,\,D_x) + \I R_1(t,\,x,\,D_x)  - \sum_{k = 1}^8 \partial_t \Phi_k(t,\,D_x)\Big) W,
	\end{aligned}
	\end{equation}
	for $(t,\,x) \in [0,\,T^\ast]\times \R^n$,
	with initial conditions
	\begin{equation*}
	\begin{aligned}
	W(0,\,x) =  \jbl D_x \jbr^{\nu} e^{-\Phi(0,\,D_x)} M^\#(0,\,x,\,D_x) H(0,\,D_x)\\\times \big(u(0,\,x),\,\ldots,\,D_t^{m-1}u(0,\,x)\big)^T,
	\end{aligned}
	\end{equation*}
	
	Recalling that $\partial_t |W|^2 = 2 \Re\big[(\partial_t W,\,W)\big]$ we obtain
	\begin{equation*}
	\begin{aligned}
	\partial_t|W|^2= 2 \Re\Big[\Big(\Big(\I  D(t,\,x,\,D_x) + \I R(t,\,x,\,D_x) + \I R_1(t,\,x,\,D_x)\\ - \sum_{k = 1}^8 \partial_t \Phi_k(t,\,D_x)\Big)W,\,W\Big)\Big].
	\end{aligned}
	\end{equation*}
	We observe that
	\begin{equation*}
	\begin{aligned}[t]
	\Re\Big[\sigma\Big(\I  D(t,\,x,\,D_x) + \I R(t,\,x,\,D_x) + \I R_1(t,\,x,\,D_x)  - \sum_{k = 1}^8 \partial_t \Phi_k(t,\,D_x)\Big)\Big]\\=	\Re\Big[\sigma\Big(\I R(t,\,x,\,D_x) + \I R_1(t,\,x,\,D_x)  - \sum_{k = 1}^8 \partial_t \Phi_k(t,\,D_x)\Big)\Big],
	\end{aligned}
	\end{equation*}
	since  $\Re[\I D(t,\,x,\,\xi)] = 0$ (by assumption~\ref{CP2:Hyperbolic}).
	In view of \eqref{PROOF:EstR}, \eqref{PROOF:EstR1} and due the construction of $\Phi$ we are able to choose the constants $M_1,\,\ldots,\,M_8$ and $\kappa$  sufficiently large, such that
	\begin{equation*}
		\I R(t,\,x,\,D_x) + \I R_1(t,\,x,\,D_x)  - \sum_{k = 1}^8 \partial_t \Phi_k(t,\,D_x),
	\end{equation*}
	is a positive operator. Thus, we are able to apply sharp G{\aa}rding's inequality to obtain
	\begin{align*}
	\partial_t |W|^2 &=
	\begin{aligned}[t]
	2 \Re\Big[\Big(\Big(\I  D(t,\,x,\,D_x) + \I R(t,\,x,\,D_x) + \I R_1(t,\,x,\,D_x)  \\- \sum_{k = 1}^8 \partial_t \Phi_k(t,\,D_x)\Big)W,\,W\Big)\Big]
	\end{aligned}\\	&\leq C | W |^2,
	\end{align*}
	which yields $
	\partial_t\|W\|_{L^2}^2 \leq C \|W\|_{L^2}^2$. Application of Gronwall's lemma leads to \begin{equation*}
	\|W(t,\,\cdot)\|_{L^2}^2 \leq C \|W(0,\,\cdot)\|_{L^2}^2,
	\end{equation*}
	for $t\in[0,\,T^\ast]$, where
	\begin{equation}\label{PROOF:ConSpace}
	\|W(0,\,\cdot)\|_{L^2}^2 =  \|\jbl D_x \jbr^{\nu} e^{-\Phi(0,\,D_x)} M^\#(0,\,x,\,D_x) U(0,\,x)\|_{L^2}^2 \leq C < \infty,
	\end{equation}
	since $M^\#$ is an operator of order zero,
	\begin{equation*}
	\sigma(-\Phi(0,\,D_x)) \sim W(\Lambda(t_\xi)) (w(\Lambda(t_\xi)))^{m-1} + \vp(\jxi),
	\end{equation*} and due to assumptions~\ref{CP2:Data} and \eqref{CP2:DefEta}.
	
	At the moment, we only local (in time) have well-posedness in spaces related to the above estimate \eqref{PROOF:ConSpace}. This concludes the proof of the local well-posedness result.
	
	For our global well-posedness result, we note that if our initial data belong to a space with more regularity than required to satisfy \eqref{PROOF:ConSpace} (which they do), we are able to apply a continuation argument as in \cite{Cicognani.2017} to get global in time well-posedness in the spaces of the initial data, with an in general infinite loss of derivatives. This concludes the proof of Theorem~\ref{CP2:Theorem}.
	
	\section{Concluding remarks} \label{Secconcluding}

	In this paper we have studied the well-posedness of weakly hyperbolic Cauchy problems with coefficients low-regular in time and smooth in space. We proposed the generalized Levi condition
	\begin{equation*}
	\big|D_x^\beta b_{m-j,\,\gamma}(t,\,x)\big| \leq C K_{|\beta|} \lambda(t)^{m-j} \big(w(\Lambda(t))\big)^{m(m-j-|\gamma|)},
	\end{equation*}
	to investigate the interplay between effects arising due to the low regularity of the coefficients and the multiplicity of the characteristic roots. We found that the influences of these effects on the weight function of the solution space are independent of each other in the sense that both effects generate a weight and just the dominate weight determines the solution space.
	
	At a first glance, the special choice of
	\begin{equation*}
	(w(\Lambda(t)))^m =  \Lambda(t)^{-\frac{s}{s-1}} \big(\log^{[\widetilde m]}(\Lambda(t)^{-1})\big)^{\widetilde \beta},
	\end{equation*}
	with $s > \frac{m-1}{m}$ might seem to be a limitation to Gevrey type Levi conditions only. However, it is possible to work with more general $w(\Lambda(t))$ as long as \ref{CP2:WAtZero} and Proposition~\ref{CP2:Remark:EstLambdaW} are satisfied.
	\begin{Example}
	Let us consider a Cauchy problem with coefficients having the modulus of continuity
	\begin{equation*}
	\mu(s) = s \Big(\log\Big(\frac{1}{s}\Big)+1\Big)^2.
	\end{equation*}
	This modulus of continuity generates the weight $\vp(\jxi) = (\log(\jxi))^2$. 
	We choose
	\begin{equation*}
		(w(\Lambda(t)))^m =  \Lambda(t)^{-1} \big(\log(\Lambda(t)^{-1})\big)^{2}.
	\end{equation*}
	In this way we get the Levi condition
	\begin{equation*}
	\big|D_x^\beta b_{m-j,\,\gamma}(t,\,x)\big| \leq C K_{|\beta|} \frac{\lambda(t)^{m-j}}{ \Lambda(t)^{m-j-|\gamma|}}  \big(\log(\Lambda(t)^{-1})\big)^{2(m-j-|\gamma|)},
	\end{equation*}
	which is between the above one and the $C^\infty$ type Levi condition
	\begin{equation*}
	\big|D_x^\beta b_{m-j,\,\gamma}(t,\,x)\big| \leq C K_{|\beta|} \frac{\lambda(t)^{m-j}}{ \Lambda(t)^{m-j-|\gamma|}}  \big(\log(\Lambda(t)^{-1})\big)^{m-j-|\gamma|}.
	\end{equation*}
	Although this choice is not covered by the main theorem of this paper, we are still able to apply the theorem for this particular choice of $w(\Lambda(t))$, since it satisfies \ref{CP2:WAtZero} and Proposition~\ref{CP2:Remark:EstLambdaW}.
	This Levi condition generates the weight
	\begin{equation*}
	W(\Lambda(t_\xi))(w(\Lambda(t_\xi)))^{m-1} \sim \log(\jxi)^2,
	\end{equation*}
	which yields well-posedness in spaces
	\begin{equation*}
	H^\nu_{\eta,\,\delta}(\R^n) = \Big\{f \in \Sw^\prime(\R^n)\,|\, \jbl D_x\jbr^\nu e^{\delta\eta(\jbl D_x \jbr)} f(x) \in L^2(\R^n)\Big\},
	\end{equation*}
	with $o(\eta(\jxi)) =  \log(\jxi)^2$.
	\end{Example}

\end{document}